\documentclass[11pt]{amsart}
\usepackage{amsfonts}
\usepackage{bbm}
\usepackage{amsfonts,amssymb,amsmath,amsthm}
\usepackage{url}
\usepackage{enumerate}
\usepackage{bbm}
\usepackage[all]{xy}
\usepackage[pdftex, colorlinks, citecolor=red, backref=page]{hyperref}
\usepackage{color,xcolor}

\newtheorem{theorem}{Theorem}[section]
\newtheorem{lemma}[theorem]{Lemma}
\newtheorem{definition}[theorem]{Definition}
\newtheorem{notion}[theorem]{}

\newtheorem{proposition}[theorem]{Proposition}
\newtheorem{corollary}[theorem]{Corollary}
\theoremstyle{definition}
\newtheorem{remark}[theorem]{Remark}
\newtheorem{question}[theorem]{Question}
\newtheorem{example}[theorem]{Example}

\author{Qingnan An}
\address{School of Mathematics and Statistics, Northeast Normal University, Changchun, {\rm130024}, China}
\email{qingnanan1024@outlook.com}

\author{Zhichao Liu}
\address{School of Mathematical Sciences,
Dalian University of Technology,
Dalian, {\rm116024}, China }
\email{lzc.12@outlook.com}


\begin{document}

\title[Total Cuntz semigroup] {A total Cuntz semigroup for $C^*$-algebras of stable rank one}

\begin{abstract}
In this paper, we show that 
for unital, separable $C^*$-algebras of stable rank one and real rank zero,  the unitary Cuntz semigroup functor and the functor ${\rm K}_*$ are naturallly equivalent. Then we introduce a refinement of the  unitary Cuntz semigroup, say the total Cuntz semigroup, which is a new invariant for separable $C^*$-algebras of stable rank one, is a well-defined continuous functor from the category of $C^*$-algebras of stable rank one to the category ${\rm\underline{ Cu}}$.  We prove that this new functor and the functor ${\rm \underline{K}}$ are naturallly equivalent for  unital, separable,  K-pure $C^*$-algebras. Therefore, the total Cuntz semigroup is a complete invariant for a large class of $C^*$-algebras of real rank zero.
\end{abstract}

\maketitle
\section*{introduction}


The Cuntz semigroup is an invariant for $C^*$-algebras that is intimately relate to Elliott's classification program for simple, separable, nuclear $C^*$-algebras.
Its original construction $W(A)$ resembles the semigroup $V(A)$ of Murray-von Neumann equivalence classes of projections, is a positively ordered, abelian semigroup whose elements are equivalent classes of positive elements in matrix algebras over $A$ \cite{Cu}. This was remedied in \cite{CEI} by constructing an ordered semigroup, termed ${\rm Cu}(A)$, in terms of countably generated Hilbert modules. The semigroup ${\rm Cu}(A)$ is order isomorphic to $W(A\otimes \mathcal{K})$, where $\mathcal{K}$ is the $C^*$-algebra of compact operators on a separable Hilbert space, it can also be regarded as a completion of $W(A)$ \cite{APT,CEI}. Moreover, a Cuntz category was described to which the Cuntz semigroup belongs and as a functor into which it preserves inductive limits. There are also many interesting consequences for the Cuntz semigroup of $C^*$-algebras of stable rank one (\cite{APRT,T}).

 The Cuntz semigroup contains a great deal of the structure of $C^*$-algebras, but the main limitation is that it fails to capture the ${\rm K}_1$ information. Recently, Cantier introduced the unitary Cuntz semigroup, denoted by Cu$_1$, as an invariant which captures the ${\rm K}_1$ information \cite{L1}, and aimed to classify certain non-simple $C^*$-algebras. A unitary Cuntz category was also defined, as an analogy to the Cuntz category.

In this paper, we first investigate the unitary Cuntz semigroup and show the close links between ${\rm Cu}_{1,u}$ and ${\rm K}_*$ for the real rank zero case, this extends Theorem 5.20 of \cite{L1}.

\begin{theorem}
Upon restriction the class of unital, separable $C^*$-algebras of stable rank one and real rank zero, the functors ${\rm Cu}_{1,u}$ and ${\rm K}_*$ are naturally equivalent.
\end{theorem}

Based on the above result, ${\rm Cu}_{1}$ is not a complete invariant for AD algebras of real rank zero, then we introduce a refinement of unitary Cuntz semigroup, denoted by ${\rm \underline {Cu}}$, and establish the functorial properties. We show that
 ${\rm \underline {Cu}}$ is a continuous functor from the category of unital, separable $C^*$-algebras of stable rank one, denoted by $C^*_{sr1}$, to a subcategory of ${\rm Cu}^\sim$. We have
\begin{theorem}
Let $A=\lim\limits_{\longrightarrow}(A_i,\phi_{ij})$ be an inductive system in the category $C^*_{sr1}$. Then
$$
{\rm \underline {Cu}-}\lim_{\longrightarrow}{\rm \underline{Cu}}(A_i)\cong
{\rm \underline{Cu}}(A).
$$
\end{theorem}

Then we show that the total K-theory can be recovered funtorially from  ${\rm \underline{Cu}}$. Furthermore, we have 
\begin{theorem}
  Upon restriction the class of unital, separable $\mathrm{K}$-$pure$ (or A$\mathcal{HD}$) algebras with stable rank one  and real rank zero, the functors  ${\rm \underline{Cu}}_{u}$ and ${\rm \underline{K}}$ are naturally equivalent. Therefore, for these algebras, ${\rm \underline{K}}$ is a classifying functor, if, and only if, so is ${\rm \underline{Cu}}_{u}$.
\end{theorem}

Note that the total K-theory  is a complete invariant for A$\mathcal{HD}$ algebras (including AD algebras and AH algebras with slow dimension growth, see \cite{GJL}) of real rank zero, then ${\rm \underline{Cu}}_u$ is also a complete invariant for this class.
The question whether all the unital, separable $C^{*}$-algebras of stable rank one and real rank zero can be classified by the total K-theory is still open.

This paper is organized as follows: In the first two sections, we present definitions and basic tools, then from a functorial point of view, we prove that the unitary Cuntz semigroup and graded K-theory determine each other for the real rank zero case.

In Section 3, we construct a total Cuntz semigroup and  prove that it satisfies the axioms (O1)--(O4). Then, we find a suitable category, called the total Cuntz category. In Section 4, we use the eventually-increasing sequence to obtain the continuity of the functor ${\rm \underline {Cu}}$ from the category of unital, separable $C^*$-algebras of stable rank one to the total Cuntz category.

Finally, we show that this new invariant performs well in classification of certain non-simple $C^*$-algebras. We can recover the total Cuntz semigroup from the total K-theory for a large class of $C^*$-algebras, but the whole relation between these invariants remains unknown.

\section{Prelimaries}

\begin{notion}\rm
  Let $A$ be a unital $\mathrm{C}^*$-algebra. $A$ is said to have stable rank one, written $sr(A)=1$, if the set of invertible elements of $A$ is dense. $A$ is said to have real rank zero, written $rr(A)=0$, if the set of invertible self-adjoint elements is dense in the set $A_{sa}$ of self-adjoint elements of $A$. If $A$ is not unital, let us denote the minimal unitization of $A$ by $A^\sim$. A non-unital $\mathrm{C}^*$-algebra is said to have stable rank one (or real rank zero) if its unitization has stable rank one (or real rank zero). Denote by $C^*_{sr1}$ the category of unital, separable $C^*$-algebras with stable rank one.
\end{notion}

\begin{notion}[\cite{Ell}]\rm
An approximatively dimension-drop algebra, written AD algebra, is an inductive limit of finite direct sums of the form $M_n(\mathbb{I}_p^\sim)$ and $M_n(C(X))$, where
$$
  \mathbb{I}_p=\{f\in C([0,1],M_p(\mathbb{C}))\,|\,f(0)=0~~{\rm and}~~f(1)\in \mathbb{C}\cdot 1_p\}
  $$
is the Elliott-Thomsen dimension-drop interval algebra
and $X$ is one of the following finite connected CW complexes: $\{pt\},~\mathbb{T},~[0, 1].$ (In \cite[5.4]{L1}, AD algebra is called ${\rm AH}_d$-algebra.)
\end{notion}

\begin{notion}\rm
  ({\bf The Cuntz semigroup of a $C^*$-algebra}) Denote the cone of positive elements in $A$ by $A_+$. Let $a,b\in A_+$. One says that $a$ is $Cuntz$ $subequivalent$ to $b$, denoted by $a\lesssim_{Cu} b$, if there exists a sequence $(r_n)$ in $A$ such that $r_n^*br_n\rightarrow a$. One says that $a$ is  $Cuntz$ $equivalent$ to $b$, denoted by $a\sim_{Cu} b$, if $a\lesssim_{Cu} b$ and $b\lesssim_{Cu} a$. The $Cuntz$ $semigroup$ of $A$ is defined as ${\rm Cu}(A)=(A\otimes\mathcal{K})_+/\sim_{Cu}$. We will denote the class of $a\in (A\otimes\mathcal{K})_+$ in ${\rm Cu}(A)$ by $[a]$. ${\rm Cu}(A)$ is a positively ordered abelian semigroup when equipped with the addition: $[ a ]+[ b] =[ a \oplus b]$, and the relation:
  $$
  [ a]\leq[ b] \Leftrightarrow a\lesssim_{Cu} b,\quad a,b\in (A\otimes\mathcal{K})_+.
  $$
\end{notion}

\begin{notion}\rm\label{Cuaxiom}
  ({\bf The category Cu})  Let $(S, \leq)$ be a positively ordered semigroup such that the suprema of increasing sequences always exists in $S$. For $x$ and $y$ in $S$, let us say that $x$ is compactly contained in $y$ (or $x$ is way-below $y$), and denote it by  $x \ll y$, if for every increasing sequence $(y_n)$ in $S$ such that $y\leq\sup _{n \in \mathbb{N}} y_{n}$, then there exists $k$ such that $x\leq y_{k} .$ This is an auxiliary relation on $S$, called the compact containment relation.
  \, If $x\in S$ satisfies $x\ll x$, we say that $x$ is compact.

We say that $S$ is a Cu-semigroup of the Cuntz category Cu, if it has a 0 element and satisfies the following order-theoretic axioms:

(O1): Every increasing sequence of elements in $S$ has a supremum.

(O2): For any $x \in S$, there exists a $\ll$-increasing sequence $\left(x_{n}\right)_{n \in \mathbb{N}}$ in $S$ such that $\sup_{n \in \mathbb{N}} x_{n}=x$.

(O3): Addition and the compact containment relation are compatible.

(O4): Addition and suprema of increasing sequences are compatible.

A Cu-morphism between two Cu-semigroups is a positively ordered monoid morphism that preserves the compact containment relation and suprema of increasing sequences.

Let $S$ be a Cu-semigroup, we say that $M$ is an Cu-ideal of $S$, if $M$ is a sub-Cu-semigroup of $S$ and order-hereditary (for any $a,b\in S$, $a\leq b$ and $b\in M$ imply $a\in M$).

\end{notion}


\begin{notion}\rm
  Let $A$ be a $C^*$-algebra and let Lat$(A)$ denote the collection of ideals in $A$, equipped with the partial order given by inclusion of ideals. For any ideal $I$ in $A$, then ${\rm Cu}(I)$ is an ideal of ${\rm Cu}(A)$. (See \cite [5.1]{APT} and \cite[3.1]{Ciu} for more details.) For any $x\in {\rm Cu}(A)$, write
  $$
  I_x=\{y\in {\rm Cu}(A)\,\,|\,\, y\leq \infty x \}
  $$
  as the ideal of ${\rm Cu}(A)$ generated by $x$, where
  $$\infty x=[\bigoplus_{n=1}^\infty \frac{1}{2^n}a] \in {\rm Cu}(A),$$
  for any $a\in (A\otimes \mathcal{K})_+$ with $[a]=x$.

  For any  $C^*$-algebra $A$, $I\rightarrow {\rm Cu}(I)$ defines a lattice isomorphism between the lattice Lat$(A)$ of closed two-sided ideals of $A$ and the lattice ${\rm Lat}({\rm Cu}(A))$ of ideals of ${\rm Cu}(A)$, denote this isomorphism by $\Gamma$, i.e.,
  $$
  \Gamma: {\rm Lat}(A)\rightarrow{\rm Lat}({\rm Cu}(A)),\quad \Gamma(I)={\rm Cu}(I).
  $$
\end{notion}

There will be an abuse of notations from now on, for any $a\in A_+$, we write $I_a$, $I_{[a]}$ for the respective ideals in $A$ and ${\rm Cu}(A)$, but we will indistinguishably use $I_a$ or $I_{[a]}$, refering to one or the other.
\begin{notion}\label{Ell inc} \rm
  Let $A=\lim(A_n,\phi_{nm})$ be an inductive limit of $C^*$-algebras. It is shown in \cite[Theorem 2]{CEI} that the Cuntz semigroup functor preserves inductive limits of sequences (this was generalized to arbitrary inductive limits in \cite[Corollary 3.2.9]{APT}, see also \cite[Lemma 3.8]{TV}) and each element $s\in {\rm Cu}(A)$ can be represented by an increasing sequence $(s_1,s_2,\cdots)$, with $s_1\in {\rm Cu}(A_1),s_2\in {\rm Cu}(A_2),\cdots$-- increasing in the sense that for each $i$, the image of $s_i\in  {\rm Cu}(A_i)$ in $ {\rm Cu}(A_{i+1})$ is less than or equal to $s_{i+1}$. Recall that two sequences $(s_i)\leq (t_i)$ if for any $i$  and any $s\in S_i$ with $s\ll s_i$, eventually $s\ll t_j$ (in $S_j$). We say $(s_i)\sim(t_i)$, if $(s_i)\leq(t_i)$ and $(t_i)\leq(s_i)$. We denote the class of $(s_i)$ in ${\rm Cu}(A)$ by $[(s_i)]$, then
  $$
   {\rm Cu}(A)= {\rm Cu}-\lim_{\longrightarrow} {\rm Cu}(A_i).
  $$

Let $(s_1,s_2,\cdots)$  be an increasing sequence with $[(s_i)]=s$, where $s_i\in {\rm Cu}(A_i)$ for any $i$. Then
$
I_s\cong{\rm Cu}-\lim\limits_{\longrightarrow} I_{s_i},
$
where $I_{s_i}$ is the ideal generated by $s_i$ in ${\rm Cu}(A_i)$. For the corresponding ideals in $C^*$-algebras, we simply write
$
I_s={\rm C^*-}\lim\limits_{\longrightarrow} I_{s_i}.
$
\end{notion}



The following theorem is \cite[Theorem 2.8]{B}.
\begin{theorem}\label{herideal}
  If $B$ is a full hereditary $C^*$-subalgebra of $A$ and if each of $A$ and $B$ has a strictly positive element, then $B$ is stably isomorphic to $A$.
\end{theorem}
 For any $a\in A_+$, the hereditary algebra her$(a)$ generated by $a$ and the ideal $I_a$ generated by $a$ are stably isomorphic, and hence, have the same K-theory. We will use this fact frequently.

\begin{notion}\rm
({\bf Unitary Cuntz semigroup})
  Recall the definition of unitary Cuntz semigroup introduced by Cantier in \cite[3.3]{L1}.
  Let $A$ be a $C^*$-algebra of stable rank one and let
$$
H(A):=\{(a,u)\,|\,a\in (A\otimes\mathcal{K})_+,u\in \mathcal{U}({{\rm her}(a)}^\sim)\}.
$$
Write $(a,u)\lesssim_1 (b,v) $ if $a\lesssim_{Cu} b$ and $[\theta_{ab,\alpha}(u)]=[v]$ in ${\rm K}_1({\rm her}(b))$, where $\theta_{ab,\alpha}$ is an explicit injection from ${\rm her}(a)$ to ${\rm her}(b)$ (see \cite[2.2]{L1}), we say $(a,u)\sim_1(b,v)$ if $(a,u)\lesssim_1 (b,v) $ and $(b,v)\lesssim_1 (a,u) $.

Define the unitary Cuntz semigroup of $A$ by
$$
{\rm Cu}_1 (A):=H(A)/\sim_1.
$$
The equivalent class of an element $(a,u)$ in $H(A)$ is denoted by $[(a,u)]$.

For any $[(a,u)], [(b,v)]\in {\rm Cu}_1(A)$, we write $[(a,u)]\leq[(b,v)]$, if $[(a,u)]\lesssim_1 [(b,v)]$, and we set $[(a,u)]+[(b,v)]=[(a\oplus b,u\oplus v)]$.

Then $({\rm Cu}_1(A),+,\leq)$ is a partially ordered monoid.

\end{notion}

Let $A$ be a $C^*$-algebra of stable rank one. Denote ${\rm Lat}_f(A)$ the subset of ${\rm Lat}(A)$ consisting of all the ideals of $A$ that contains a full positive element. Hence, any ideal in ${\rm Lat}_f(A)$ is singly-generated.
Define {\rm Cu}$_f(I):=\{[ a]\in {\rm Cu}(A)\,|\, I_a=I\}$. In other words, ${\rm Cu}_f(I)$ consists of the elements of ${\rm Cu}(A)$ that are full in ${\rm Cu}(I)$.

The following theorem illustrates the main structure of ${\rm Cu}_1(A)$.

\begin{theorem}{\rm (}\cite[Theorem 4.8]{L1}{\rm )}\label{mainstr}
  Let $A$ be a separable $C^*$-algebra of stable rank one, then there is a ${\rm Cu}^\sim$-isomorphism:
  $${\rm Cu}_1(A)\cong \coprod_{I\in {\rm Lat}_f(A)}{\rm Cu}_f(I)\times {\rm K}_1(I).$$
  The addition and order on $\coprod_{I\in {\rm Lat}_f(A)}{\rm Cu}_f(I)\times {\rm K}_1(I)$ are defined as follows:
  For any $(x, k) \in \mathrm{Cu}_{f}\left(I_{x}\right) \times \mathrm{K}_{1}\left(I_{x}\right)$ and $(y, l) \in \mathrm{Cu}_{f}\left(I_{y}\right) \times \mathrm{K}_{1}\left(I_{y}\right)$,
$$
\left\{\begin{array}{l}
(x, k) \leq(y, l), \,\,{if }\,\, x \leq y \,\, {  and }
\,\,\delta_{I_{x} I_{y}}(k)=l, \\
(x, k)+(y, l)=\left(x+y, \delta_{I_{x} I_{x+y}}(k)+\delta_{I_{y} I_{x+y}}(l)\right) ,
\end{array}\right.
$$
where $\delta_{IJ}={\rm K}_1 (I\xrightarrow{i} J)$, for any $I,J\in{\rm Lat}(A)$ such that $I\subset J$.
\end{theorem}
If there is an $I\in{\rm Lat}(A)\setminus{\rm Lat}_f(A)$, we have ${\rm Cu}_f(I)=\varnothing$.
We mention that whenever $A$ is separable, then ${\rm Lat}(A)={\rm Lat}_f(A)$.
\begin{notion}\label{RE AND CU}\rm \cite[5.5]{APT}
Denote Mon$_\leq$ the category of ordered monoids and POM the category consisting of positively ordered monoids. Given  $M\in$ POM, denote its Cu-completion by Cu$(M)$ (the set of any increasing sequence in $M$ modulo equivalent relation) (see more details in \cite[3.1.6]{APT}), Thus, we obtain a functor
$$
{\rm Cu}:\, {\rm POM}\to {\rm Cu},\quad M\longmapsto {\rm Cu}(M),\quad {\rm for~~ all~~}M\in{\rm POM},
$$
mapping the POM-morphism $f:\,M\to N$ into the induced Cu-morphism ${\rm Cu}(f):\,{\rm Cu}(M)\to {\rm Cu}(N)$.

Conversely, given a Cu-semigroup $S$, we denote by $S_c$ the set of compact elements in $S$.
It is easy to see that $S_c$ is a submonoid of $S$ and we equip it with the order induced by $S$. Then we get another functor
$${\mathfrak{C}}:\,  {\rm Cu}\to{\rm POM},\quad S\longmapsto S_c,\quad {\rm for~~ all~~}S\in {\rm Cu},
$$

It is obvious that the two functors Cu and ${\mathfrak{C}}$ establish an equivalence of the category POM and the full subcategory of Cu consisting of algebraic Cu-semigroups (every
element is the supremum of an increasing sequence of compact elements).
\end{notion}
\begin{notion}\rm
  The $unitary\,\,Cuntz\,\, category$, written ${\rm Cu}^\sim$, is the subcategory of Mon$_\leq$ whose objects are ordered monoids satisfying the axioms (O1)- (O4) and such that $0\ll 0$ and morphisms are Mon$_\leq$-morphisms that respect suprema of increasing sequences and the compact containment relation.

  Let $S,\,T\in {\rm Cu}^\sim$ and $f:\,S\rightarrow T$ be a ${\rm Cu}^\sim$-morphism,
then $S_c$, $T_c$ are ordered submonoids of $S,~T$, respectively.
  Denote the restriction map of $f$ by $f_c:\,S_c\rightarrow T_c$, and denote the induced Grothendieck map of $f_c$ by
  $$
  {\rm Gr}(f_c): {\rm Gr}(S_c)\rightarrow {\rm Gr}(T_c).
  $$
\end{notion}
\begin{definition}\rm 
Give $S\in {\rm Cu}^{\sim}$, we say that $S$ is an $algebraic$ ${\rm Cu}^{\sim}$-semigroup if every
element in $S$ is the supremum of an increasing sequence of compact elements, 
that is, an increasing sequence in $S_c$. 
\end{definition}

\begin{corollary}{\rm (}\cite[Corollary 3.25]{L1}{\rm )}\label{algebraic cor}
Let $A\in C^*_{sr1}$. Then $A$ has real rank zero if and only if ${ \rm Cu}_1(A)$ is algebraic if and only if ${ \rm Cu}(A)$ is algebraic.
\end{corollary}
\begin{definition}
  Let $\mathcal{C},\mathcal{D}$ be arbitrary categories, and let $I:C^*\rightarrow \mathcal{C}$ and $J: C^*\rightarrow \mathcal{D}$ be covariant functors. Let $H:\mathcal{C}\rightarrow \mathcal{D}$ be a functor such that there exists a natural isomorphism $\eta :H\circ J \simeq I$. Then we say we can recover $I$ from $J$ through $H$.
\end{definition}

\section{${\rm Cu}_1$ and ${\rm K}_*$}

Cantier has shown that for any unital, separable $C^{*}$-algebras of stable rank one, the invariant ${\rm K}_*$ can be recovered from ${\rm Cu}_1$, in this section, under the real rank zero setting,  we prove the converse.
\begin{definition}
  \cite[Definition 1.2.1]{EG} \rm Let $A$ be a (unital) $\mathrm{C}^{*}$-algebra and set $\mathrm{K}_{*}(A)=\mathrm{K}_{0}(A) \oplus\mathrm{K}_{1}(A).$ Define
  $\mathrm{K}_{*}^{+}(A)=\{ ([p],[u \oplus\left(1_{k}-p\right)]\}$, where $p \in M_{k}(A)$ is a projection, $u \in p M_{k}(A) p$ is a unitary and $\mathbf{1}_{k} \in M_{k}(A)$ is the unit. Note that we also use $[u]$ or $[u]_{{\rm K}_1(A)}$ to denote the ${\rm K}_1$ class  $[u+(1_k-p)]_{{\rm K}_1(A)}$.
\end{definition}

The following proposition is well-known, see \cite[Proposition 4]{LR}.
\begin{proposition}\label{lin inj}
Let
$$
0 \longrightarrow A \longrightarrow E \longrightarrow B \longrightarrow 0
$$
be a short exact sequence of $\mathrm{C}^{*}$-algebras. Let $\delta_{j}: \mathrm{K}_{j}(B) \rightarrow \mathrm{K}_{1-j}(A)$ for $j=0,1$, be the index maps of the sequence.

(i) Assume that $A$ and $B$ have real rank zero. Then the following three conditions are equivalent :

(a) $\delta_{0} \equiv 0$,

(b) $rr(E)=0$,

(c) all projections in $B$ are images of projections in $E$.

(ii) Assume that $A$ and $B$ have stable rank one. Then the following are equivalent:

(a) $\delta_{1} \equiv 0$,

(b) $sr(E)=1$.

If, in addition, $E$ (and $B$ ) are unital then (a) and (b) in (ii) are equivalent to

(c) all unitaries in $B$ are images of unitaries in $E$.

\end{proposition}

\begin{example}\label{example sur not in}\rm
 For any $C^*$-algebra $A$ of stable rank one, ${\rm  Cu}(A)$  has weak cancellation (see \cite{E},\,\cite{RW}), i.e.,  $x+z\ll y+z$ implies $x\leq y$ for $x,y,z\in {\rm  Cu}(A)$, but this is not true for ${\rm Cu}_1(A)$, we present an example here.

 Set
  $$
  E=\{\,f\in M_2(C[0,1])\,|\,f(0)=\left(\begin{array}{cc}
                                    \lambda & 0 \\
                                    0 & \mu
                                  \end{array}\right)
\,{\rm and}\, f(1)=\left(\begin{array}{cc}
                                    \nu & 0 \\
                                    0 & \mu
                                  \end{array}\right)
,\lambda,\mu,\nu\in \mathbb{C}\}.
$$
We have the following short exact sequence
$$
0\rightarrow M_2(C_0(0,1)) \xrightarrow{\iota} E \xrightarrow{\pi} \mathbb{C}\oplus \mathbb{C}\oplus \mathbb{C}\rightarrow0,
$$
where $\iota$ is the natural embedding map and $\pi(f)=(\lambda,\nu,\mu)$ for $f\in E$.
Then one has the six-term exact sequence
$$
0\to \mathrm{K}_0(E)\xrightarrow{\pi_*} \mathbb{Z}\oplus \mathbb{Z}\oplus \mathbb{Z} \xrightarrow{\delta_0} \mathbb{Z}\xrightarrow{\iota_*} \mathrm{K}_1(E)\to 0,
$$
where $\delta_0=(1,-1,0)$.

Such an $E$ is an Elliott-Thomsen algebra, one can see \cite{GLN} for more details. Here, we have ${\rm K}_1(E)=0$ ($\delta_0$ is surjective, and hence, $\iota_*=0$) and $sr(E)=1$ ($E$ is an extension of two stable rank one algebras and $\delta_1\equiv 0 :{\rm K}_1(\mathbb{C}\oplus \mathbb{C}\oplus \mathbb{C})\rightarrow {\rm K}_0(M_2(C_0(0,1))) $, see Proposition \ref{lin inj} (ii)).

Let $q=\left(\begin{array}{cc}
                                    0 & 0 \\
                                    0 & 1
                                  \end{array}\right)\in E$  be a projection,
then ${\rm her}(q)\cong C(\mathbb{T})$ and  ${\rm K}_1(I_{q})={\rm K}_1({\rm her}(q))={\rm K}_1(C(\mathbb{T}))=\mathbb{Z}$.

Take
$$
([q],0),([q],1)\in {\rm Cu}_1(E).
$$
Since the natural embedding map from $I_{q}$ to $E\otimes \mathcal{K}$ induces $\delta_{I_{q}E}:{\rm K}_1(I_{q})\rightarrow {\rm K}_1(E)=0$, then
$$
([q],0)+([1_E],0)=([q]+[1_E],\delta_{I_qE}(0)+0)= ([q]+[1_E],0),
$$
$$
([q],1)+([1_E],0)=([q]+[1_E],\delta_{I_qE}(1)+0)= ([q]+[1_E],0).
$$
But $([q],0)\neq([q],1).$

This means ${\rm Cu}_1(E)$ doesn't satisfy the cancellation of compact elements in the sense of \cite[Proposition 5.16]{L1}. The basic reason is that the natural map $\delta_{I_{q}E}=\mathrm{K}_1(\iota):\,\mathrm{K}_1(I_{q})\to \mathrm{K}_1(E)$ is not injective, though $\iota:\,I_{q}\to E$ itself is injective and
$\mathrm{K}_0(\iota):\,\mathrm{K}_0(I_{q})\to \mathrm{K}_0(E)$ is injective.

Moreover, consider the following map
\begin{eqnarray*}
            \alpha: {\rm Cu}_1(A)_c &\rightarrow & {\rm K}_*^+(A) \\
            ([q],[u]) &\mapsto & ([q],\delta_{I_qA}([u]))
\end{eqnarray*}
In general, $\alpha$ is just a surjective map, not an isomorphism (such as $A=E$). This means the proof of  \cite[Theorem 5.20]{L1} is not entirely correct. We point out that the conclusion is still true, as one can still prove
$
({\rm Gr}(\mathrm{Cu}_1(A)_c),$ $\rho(\mathrm{Cu}_1(A)_c))\cong(\mathrm{K}_*(A),\mathrm{K}_*^+(A))
$
without the injectivity of $\alpha$, where $\rho:\, \mathrm{Cu}_1(A)_c\to {\rm Gr}(\mathrm{Cu}_1(A)_c)$ is the natural map ($\rho(x)=[(x,0)])$.
We will write down a proof and also a refinement version. 

\end{example}

The following result is partially proved in \cite[Theorem 5.20]{L1}, but there is one deficient part left for the proof, now we present our method.

\begin{theorem} \label{cptlem}
Let $A$ be a unital, separable $C^{*}$-algebra with stable rank one. Then
$$
({\rm Gr}(\mathrm{Cu}_1(A)_c),\rho(\mathrm{Cu}_1(A)_c))\cong(\mathrm{K}_*(A),\mathrm{K}_*^+(A)).
$$
Moreover, if $A$ is of real rank zero, then
$$
({\rm Cu}_1(A)_c, ([1_A],0))\cong ({\rm K}_*^+(A),[1_A]).
$$
\end{theorem}
\begin{proof}
Let $[(a, u)]$ be a compact element in  $\mathrm{Cu}_{1}(A)$, by \cite[Corollary 3.5]{L1},  $[a]$ is a compact element of $\mathrm{Cu}(A)$. Since $A$ has stable rank one, use \cite[Theorem 5.8]{BC}, there exists a projection $p \in A\otimes \mathcal{K}$ such that $[p]=[a]$. Then every compact element in $\mathrm{Cu}_{1}(A)$ can be written as $[(p,u)]$, where $p$ is a projection in $A\otimes \mathcal{K}$ and $u$ is a unitary element in her$(p)$.
From the view of Theorem \ref{mainstr}, $[(p,u)]$ can also be written as $([p],[u]_{\mathrm{K}_{1}(I_p)}).$
(${\mathrm{K}_{1}(I_p)}$ and $\mathrm{K}_{1}(\operatorname{her}(p))$ are naturally isomorphic.)
Then the map
\begin{eqnarray*}
            \alpha: {\rm Cu}_1(A)_c &\rightarrow & {\rm K}_*^+(A) \\
          {([p],[u]_{\mathrm{K}_{1}(I_p)})} & \mapsto & ([p],\delta_{I_pA}([u]_{\mathrm{K}_{1}(I_p)}))
\end{eqnarray*}
is a monoid morphism, where
$
\delta_{I_{p} A}: \mathrm{K}_{1}(I_p) 
{\longrightarrow} \mathrm{K}_{1}(A).
$

Note that for any $([q],[v]_{{\rm K}_1(A)})\in {\rm K}_*^+(A)$, where $q$ is a projection in $A\otimes\mathcal{K}$ and $v$ is a unitary in her$(q)$, $([q],[v]_{{\rm K}_1(I_q)})$ is a compact element in ${\rm Cu}_1(A)_c$. By the inclusion from $I_q$ to $A\otimes\mathcal{K}$, we have $\delta_{I_{q} A}([v]_{{\rm K}_1(I_q)})=[v]_{{\rm K}_1(A)}.$ Then $\alpha$ is surjective.

Suppose that
$$
\alpha([p],[u]_{{\rm K}_1(I_p)})=\alpha([q],[v]_{{\rm K}_1(I_q)}).
$$
Then for $([1_A],0)\in {\rm Cu}_1(A)_c$, we have
$$
([p]+[1_A],\delta_{I_eA}([u]_{{\rm K}_1(I_e)}))=([q]+[1_A],\delta_{I_qA}([v]_{{\rm K}_1(I_q)})),
$$
which is
$$
([p],[u]_{{\rm K}_1(A)})+([1_A],0)=([q],[v]_{{\rm K}_1(A)})+([1_A],0).
$$
Hence, the images of $([p],[u]_{{\rm K}_1(A)})$ and $([q],[v]_{{\rm K}_1(A)})$ in ${\rm Gr}({\rm Cu}_1(A)_c)$ are the same.

 Consider the natural map $\rho:\, \mathrm{Cu}_1(A)_c\to {\rm Gr}(\mathrm{Cu}_1(A)_c)$, if $\rho(x)=\rho(y)$, then there exists $z\in \mathrm{Cu}_1(A)_c$ such that
$x+z=y+z$, hence, $$\alpha(x)+\alpha(z)=\alpha(y)+\alpha(z).$$ Recall that $({\rm K}_*(A),{\rm K}_*^+(A))$ is an ordered group, then $\alpha(x)=\alpha(y)$.

Now we have
$$
{\rm Gr}({\rm Cu}_1(A)_c)\cong{\rm Gr}(\alpha({\rm Cu}_1(A)_c))={\rm Gr}({\rm K}_*^+(A))
$$
and
$$
({\rm Gr}(\mathrm{Cu}_1(A)_c),\rho(\mathrm{Cu}_1(A)_c))\cong(\mathrm{K}_*(A),\mathrm{K}_*^+(A)).
$$

Generally speaking, $\alpha$ is not always injective,  but under the assumption of real rank zero, this is true.

Suppose that
$$
\alpha([p],[u]_{{\rm K}_1(I_p)})=\alpha([q],[v]_{{\rm K}_1(I_q)}).
$$
Then $p\sim_{Cu}q$, and hence, $I_p=I_q$. By Lemma \ref{lin inj}, we have $\delta_{I_pA}$ and $\delta_{I_qA}$ are the same injective map, which means we can identify $[u]_{{\rm K}_1(I_p)},$ $[u]_{{\rm K}_1(A)}$, $[v]_{{\rm K}_1(A)}$  and $[v]_{{\rm K}_1(I_q)}$ as a same element.

Since $A$ has stable rank one, Murray von Neumann equivalence and Cuntz equivalence agree on projections. That is, $([p],[u]_{{\rm K}_1(I_p)})=([q],[v]_{{\rm K}_1(I_q)})$, and hence, $\alpha$ is injective. Therefore, $\alpha$ is an isomorphism.

Note that ${\rm K}_0^+(A)$ is a sub-cone of ${\rm K}_*^+(A)$, it induces an order on ${\rm K}_*^+(A)$, hence, in ${\rm K}_*^+(A)$ we say
$$
([p],[u]_{{\rm K}_1(A)})\leq ([q],[v]_{{\rm K}_1(A)})\,\,{\rm iff}\,\, [p]\leq [q],\,\,[u]_{{\rm K}_1(A)}=[v]_{{\rm K}_1(A)}.
$$
Then $\alpha$ becomes an ordered isomorphism from ${\rm Cu}_1(A)_c$ to $\mathrm{K}^{+}_{*}(A)$.

Note that $\mathrm{Cu}_{1}(A)$ is algebraic and completely determined by $\mathrm{Cu}_{1}(A)_{c}$, meanwhile, $\mathrm{K}_{*}(A)$ is the Grothendieck group of $\mathrm{K}^{+}_{*}(A)$. Moreover, $\alpha$  can be regarded as an ordered isomorphism:
$$
({\rm Cu}_1(A)_c, ([1_A],0))\cong ({\rm K}_*^+(A),[1_A]).
$$
This completes the proof.

\end{proof}

\begin{proposition}\label{weakcancel}
   Let $A$ be a unital, separable $C^*$-algebra of stable rank one and real rank zero. Then ${\rm  Cu}_1(A)$  has weak cancellation, i.e.,  $x+z\ll y+z$ implies $x\leq y$ for $x,y,z\in {\rm  Cu}_1(A)$.
\end{proposition}
\begin{proof}
For each  ideal $I$ of $A$, by Proposition \ref{lin inj}, we will identify ${\rm K}_1(I)$ with its natural image in ${\rm K}_1(A)$.

  Suppose that
  $$
  ([a],[u]_{{\rm K}_1(I_a)})+([c],[w]_{{\rm K}_1(I_c)})\ll([b],[v]_{{\rm K}_1(I_b)})+([c],[w]_{{\rm K}_1(I_c)}),
  $$
  then from \cite[Proposition 4.3]{L1}, we have
  $$
  [a]+[c]\ll [b]+[c]\,\,\,\,({\rm in\,\,Cu}(A))
  $$
  and
  $$
  [u]_{{\rm K}_1(I_a)}+[w]_{{\rm K}_1(I_c)}= [v]_{{\rm K}_1(I_b)}+[w]_{{\rm K}_1(I_c)}.
  $$

  By the weak cancellation of ${\rm Cu}(A)$ (see \ref{example sur not in}),
  we have $[a]\leq [b]$.  Note that  we also have
  $$
  [u]_{{\rm K}_1(I_a)}= [v]_{{\rm K}_1(I_b)}.
  $$
  Therefore,
  $$
  ([a],[u]_{{\rm K}_1(I_a)})\leq([b],[v]_{{\rm K}_1(I_b)}).
  $$
\end{proof}
\begin{remark}
   In general, weak cancellation may not hold for algebraic Cu-semigroup (Cu$^\sim$-semigroup). For each $k\geq 1$, the semigroup $E_k=\{0,1,\cdots,k,\infty\}$ with the natural order and with $a+b$ defined as $\infty$ if usually one would have  $a+b\geq k+1$ is an algebraic Cu-semigroup (Cu$^\sim$-semigroup), and all elements are compact (see \cite[5.1.16]{APT}). Then $E_k$ doesn't have weak cancellation.
\end{remark}
\begin{notion}\rm (\cite[3.9]{L1})\label{cucomplete}
  Let $S\in {\rm Mon}_{\leq}$, denote ${\rm Cu}^\sim(S):= \gamma^\sim(S,\leq)$ the ${\rm Cu}^{\sim}$-completion of $(S,\leq)$, then the assignment ${\rm Cu}^{\sim}$ from an ordered monoid $S$ to ${\rm Cu}^\sim(S)$ is a functor. Denote $\iota_S:\,S\rightarrow {\rm Cu}^\sim(S)$ the natural embedding map.
\end{notion}
\begin{proposition}{\rm (}\cite[Proposition 3.23]{L1}{\rm )}\label{alg completion}

  (i)Let $M\in {\rm Mon}_{\leq}$. Then ${\rm Cu}^\sim(M)$ is an algebraic ${\rm Cu}^\sim $-semigroup and, moreover, there is a natural identification between $M$ and the order monoid of compact elements of ${\rm Cu}^\sim(M)$.

  (ii)For any algebraic ${\rm Cu}^\sim$-semigroup $S$, we have  ${\rm Cu}^\sim(S_c)\cong S$ naturally as ${\rm Cu}^\sim$-semigroups.
\end{proposition}
\begin{notion}\label{sf property}\rm
 Let $S$ be a $\mathrm{Cu}^{\sim}$-semigroup. We say that $S$ is positively directed, if for any $x \in S$, there exists $p_{x} \in S$ such that $x+p_{x} \geq 0$.

Consider the semigroup $S=\mathbb{Z}\cup\{\infty\}$ with the natural order. Then $S$ is a positively directed algebraic Cu$^\sim$-semigroup with $S_c=\mathbb{Z}$, but
$$\rho(S_c)\cap \{-\rho(S_c)
 \}=\mathbb{Z}\neq \{0\},$$
where $\rho:\, S_c\to {\rm Gr}(S_c)$ is the natural map ($\rho(x)=[(x,0)])$.
Hence, $({\rm Gr}(S_c),\rho(S_c))=(\mathbb{Z},\mathbb{Z})$ is not an ordered abelian group.

For any
 separable unital  $C^{*}$-algebra  $A$ of stable rank one, on one hand, $\mathrm{Cu}_{1}(A)$ is positively directed \cite[Lemma 5.2]{L1}; on the other hand, from the facts that
 $\mathrm{K}_*^+(A)\cap\{-\mathrm{K}_*^+(A)\}=\{0\}
$ (\cite[ 1.2.2]{EG})
and that
 $$
({\rm Gr}(\mathrm{Cu}_1(A)_c),\rho(\mathrm{Cu}_1(A)_c))\cong(\mathrm{K}_*(A),\mathrm{K}_*^+(A)),$$
we have $\rho (\mathrm{Cu}_{1}(A)_c)\cap\{-\rho (\mathrm{Cu}_{1}(A)_c)\}=\{0\}$.
\end{notion}
\begin{notion}\rm \label{deffunctors}
The category of ordered groups with ordered unit, written AbGp$_u$, is the category whose objects are ordered groups with order-unit and morphisms are ordered group morphisms that preserve the order-unit.

Let $S$ be a positively directed $\mathrm{Cu}^{\sim}$-semigroup satisfying
$\rho(S_c)\cap \{-\rho(S_c)
 \}=\{0\}.$
Also suppose that $S_{+}$ admits a compact order-unit. We say that $(S, u)$ is a $\mathrm{Cu}^{\sim}$-$semigroup\,\,with\,\,compact \,\,order$-$unit$. Now, a $\mathrm{Cu}^{\sim}$-morphism preserves the order-unit between two $\mathrm{Cu}^{\sim}$-semigroups with compact order-unit $(S, u),(T, v)$ is a $\mathrm{Cu}^{\sim}$-morphism $\alpha: S \longrightarrow T$ such that $\alpha(u) \leq v$.

Denote ${\rm Cu}_u^\sim$ the category whose objects are $\mathrm{Cu}^{\sim}$-semigroups with compact order-unit and morphisms are ${\rm Cu}^\sim$-morphisms that preserve the order-unit.

Recall the functors defined in \cite[ 5.4]{L1},
 \begin{eqnarray*}
    {\rm K}_* :\,{\rm C}_{sr1}^*&\rightarrow&{\rm AbGp}_u \\
    A &\mapsto&({\rm K}_*(A),{\rm K}_*^+(A), [1_{A}]) \\
    \phi &\mapsto& {\rm K}_*(\phi).
  \end{eqnarray*}
\begin{eqnarray*}
    {\rm Cu}_{1,u} :\, {\rm C}_{sr1}^*&\rightarrow&{\rm Cu}_u^\sim\\
    A&\mapsto&({\rm Cu}_1(A),([1_A],0)) \\
    \phi &\mapsto& {\rm Cu}_1(\phi).
  \end{eqnarray*}\begin{eqnarray*}
    H_* :\,{\rm Cu}_u^\sim&\rightarrow&{\rm AbGp}_u \\
    (S,u) &\mapsto&({\rm Gr}(S_c),\rho(S_c),u) \\
    \alpha &\mapsto& {\rm Gr}(\alpha_c).
  \end{eqnarray*}
 Note that if we restrict the domain of $H_*$ to the full subcategory of unitary Cuntz semigroups of separable, unital $C^*$-algebras of stable rank one and real rank zero,  then Proposition \ref{weakcancel} and \ref{sf property}  imply that $H_*$ is a faithful functor. (The proof of \cite[Lemma 5.19]{L1} holds in this case.)
  \end{notion}

  \begin{theorem}\label{Kthm}
   Upon restriction to the class of unital, separable $C^*$-algebras of stable rank one and real rank zero, there are natural equivalences of functors:
  $$
  H_*\circ {\rm Cu}_{1,u}\simeq {\rm K}_*\quad{ and}\quad G\circ {\rm K}_*\simeq {\rm Cu}_{1,u}.
  $$
\end{theorem}
\begin{proof}
  We restrict the domain of $H_*$ to the full subcategory whose objects are unitary Cuntz semigroup of separable $C^*$-algebras of stable rank one and real rank zero together with compact order-unit, while the codomain  of $H_*$ is the category of the ${\rm K}_*$ of the same class. We will prove that $H_*$ yields an equivalence of functors between ${\rm Cu}_{1,u}$ and ${\rm K}_*$ under our assumption. We only need to show that the restriction functor of
  $H_*$, which we will still call $H_*$,
  is a full, faithful and dense functor.

  From the last statement in \ref{deffunctors}, $H_*$ is faithful. For any unital, separable $C^*$-algebra $A$ of stable rank one and real rank zero, 
  denote the canonical isomorphism we obtained from
   Lemma \ref{cptlem} as
  $$ \alpha_{A}^*:\,
({\rm Gr}(\mathrm{Cu}_1(A)_c),\rho(\mathrm{Cu}_1(A)_c))\cong(\mathrm{K}_*(A),\mathrm{K}_*^+(A)),
  $$
  which forms
  $$\alpha_A^*:\,
  H_*({\rm Cu}_1(A)_c, ([1_A],0))\cong(\mathrm{K}_*(A),\mathrm{K}_*^+(A), [1_A]).
  $$
  This means that $H_*$ is dense. It remains to show that $H_*$ is full.

  Since $A$ has stable rank one and real rank zero, ${\rm K}_*^+(A)$ is a subset in ${\rm K}_*(A)$, then for any order morphism
  $$
  \xi : ({\rm K}_*(A),{\rm K}_*^+(A), [1_{A}])\rightarrow ({\rm K}_*(B),{\rm K}_*^+(B), [1_{B}])
  $$
  with $\xi|_{{\rm K}_*^+(A)}:{\rm K}_*^+(A)\rightarrow {\rm K}_*^+(B)$, $\xi([1_A])\leq [1_B]$ and
  $$
  \xi|_{{\rm K}_*^+(A)}({\rm K}_0^+(A))\subset {\rm K}_0^+(B).
  $$

  By the functoriality of ${\rm Cu}^\sim$ (see \ref{cucomplete}), $\xi|_{{\rm K}_*^+(A)}$ induces an ordered ${\rm Cu}^\sim$-morphism
  $$
  \gamma^\sim(\xi|_{{\rm K}_*^+(A)}):({\rm Cu}^\sim({\rm K}_*^+(A)), ([1_A],0))\rightarrow({\rm Cu}^\sim({\rm K}_*^+(B)), ([1_B],0)),
  $$
  where the order on ${\rm K}_*^+(A)$ is induced by ${\rm K}_0^{+}(A)$.

  By Proposition \ref{alg completion} and Lemma \ref{cptlem}, we have
  $$
  {\rm Cu}_1(A)\cong {\rm Cu}^{\sim}({\rm Cu}_1(A)_c)
  \cong{\rm Cu}^{\sim}({\rm K}_*^+(A))
  $$
  and
  $$
({\rm Cu}_1(A)_c, ([1_A],0))\cong ({\rm K}_*^+(A),[1_A]).
$$
  Denote $i_A :\,{\rm Cu}_1(A)\rightarrow{\rm Cu}^{\sim}({\rm K}_*^+(A))$ the canonical ${\rm Cu}^\sim$-isomorphism and denote $
\alpha_A:\,({\rm Cu}_1(A)_c, ([1_A],0))\to ({\rm K}_*^+(A),[1_A])
$ the canonical ordered monoid isomorphism.

Note that
  $$
  i_B^{-1} \circ\gamma^\sim(\xi|_{{\rm K}_*^+(A)})\circ i_A:({\rm Cu}_1(A), ([1_A],0))\rightarrow({\rm Cu}_1(B), ([1_B],0))
  $$
  is an ordered ${\rm Cu}^\sim$-morphism.
  After identifying ${\rm Cu}_1(A)_c$ and ${\rm Cu}_1(B)_c$ with ${\rm K}_*^+(A)$ and ${\rm K}_*^+(B)$
through $\alpha_A$ and $ \alpha_B$, respectively, we have
  $$
  (i_B^{-1} \circ\gamma^\sim(\xi|_{{\rm K}_*^+(A)})\circ i_A)_c=\xi|_{{\rm K}_*^+(A)}.
  $$
Using the functoriality of $H_*$, then
  $$
  {\rm Gr}((i_B^{-1} \circ\gamma^\sim(\xi|_{{\rm K}_*^+(A)})\circ i_A)_c): ({\rm K}_*(A),{\rm K}_*^+(A),[1_{A}])\rightarrow ({\rm K}_*(B),{\rm K}_*^+(B), [1_{B}])
  $$
  is an ordered morphism. Now we have 
  $$
  H_*(i_B^{-1} \circ\gamma^\sim(\xi|_{{\rm K}_*^+(A)})\circ i_A)={\rm Gr}((i_B^{-1} \circ\gamma^\sim(\xi|_{{\rm K}_*^+(A)})\circ i_A)_c)={\rm Gr}(\xi|_{{\rm K}_*^+(A)})= \xi.
  $$
  Then $H_*$ is full.

  Therefore, by standard category theory, there exists a functor $G$ such that $H_*\circ G$ and $G\circ H_*$ are naturally equivalent to the respective identities. Then we have
  $$
  H_*\circ {\rm Cu}_{1,u}\simeq {\rm K}_*\quad{\rm and}\quad G\circ {\rm K}_*\simeq {\rm Cu}_{1,u}.
  $$
\end{proof}
\begin{corollary}
\label{Kthm1}
Let $A,\,B$ be unital, separable $C^{*}$-algebra with stable rank one and real rank zero.
Then
$({\rm Cu}_1(A),([1_A],0))\cong({\rm Cu}_1(B),([1_A],0))$
if, and only if,
$({\rm K}_*(A),{\rm K}_*^+(A),[1_A])\cong ({\rm K}_*(B),{\rm K}_*^+(B),[1_B]).$
\end{corollary}
\begin{remark}
   It was shown in \cite[Example 2.19]{EGS} that ${\rm K}_*$ is not a complete invariant for AD algebras  of real rank zero (see also \cite[Theorem 3.3]{DL3}). Then Corollary 5.21 in \cite{L1} the author obtained as the classification theorem for real rank zero AD algebras (which is called AH$_d$ algebras in \cite{L1}) in terms of unitary Cuntz semigroup, is not entirely correct.
   He used Theorem 7.3 of \cite{Ell} to get the corollary. However it was pointed out by Elliott-Gong-Su \cite{EGS} that, Theorem 7.3 (and Theorem 7.1) of \cite{Ell} is only true for simple case. 
   Then ${\rm Cu}_1$ is not a complete  invariant for AD algebras (${\rm AH}_d$ algebras).
\end{remark}


\begin{remark}
It was also proved that in \cite{EGS} the total K-theory is a complete invariant for AD algebras of real rank zero. The total K-theory works quite well in the classification of non-simple $C^*$-algebras (see \cite{G,DG,AELL}). Then it is necessary to give a generalized version of Cuntz semigroup for ${\rm \underline{K}}$ just like ${\rm Cu}_{1}$ for ${\rm K}_*$.
\end{remark}

\section{The total Cuntz semigroup}
In this section, we introduce the total Cuntz semigroup, which is a refinement of the unitary Cuntz semigroup. We show that ${\rm \underline{ Cu}}$ is a functor from the category of unital, separable $C^*$-algebras of stable rank one to the total Cuntz category.


\begin{notion}\label{def k-total}\rm
 {\bf (The total K-theory) \cite[Section 4]{DG}.} For $n\geq 2$, the mod-$n$ K-theory groups are defined by
$$\mathrm{K}_* (A;\mathbb{Z}_n)=\mathrm{K}_*(A\otimes C_0(W_n)),$$
where $W_n$ denotes the Moore space obtained by attaching the unit disk to the circle by a degree $n$-map, such as $f_n: \mathbb{T}\rightarrow \mathbb{T}$, $e^{{\rm i}\theta}\mapsto e^{{\rm i}n\theta}$. The $C^*$-algebra $C_0(W_n)$ of continuous functions vanishing at the base point is isomorphic to the mapping cone of the canonical map of degree $n$ from $C(\mathbb{T})$ to itself.

 In the setting of \cite{S}, the mod-$n$ $K$-theory groups are defined by
$$
{\rm K}_{i}(A ; \mathbb{Z}_n)={\rm K}_{i}\left(A \otimes C_{0}\left(W_{n}\right)\right),\,\,\,\,i=0,1.
$$

Let ${\rm K}_{*}(A ; \mathbb{Z}_n)={\rm K}_{0}(A ; \mathbb{Z}_n) \oplus {\rm K}_{1}(A ; \mathbb{Z}_n)$. For $n=0$, we set ${\rm K}_{*}(A ; \mathbb{Z}_n)=$ ${\rm K}_{*}(A)$ and for $n=1, {\rm K}_{*}(A ; \mathbb{Z}_n)=0$.

For a $\mathrm{C}^{*}$-algebra $A$, one defines the total K-theory of $A$ by
$$
\underline{{\rm K}}(A)=\bigoplus_{n=0}^{\infty} {\rm K}_{*}(A ; \mathbb{Z}_n) .
$$
It is a $\mathbb{Z}_2 \times \mathbb{Z}^{+}$graded group. It was shown in \cite{S} that the coefficient maps
$$
\begin{gathered}
\rho: \mathbb{Z} \rightarrow \mathbb{Z}_n, \quad \rho(1)=[1], \\
\kappa_{mn, m}: \mathbb{Z}_m \rightarrow \mathbb{Z}_{mn}, \quad \kappa_{m n, m}[1]=n[1], \\
\kappa_{n, m n}: \mathbb{Z}_{mn} \rightarrow \mathbb{Z}_n, \quad \kappa_{n, m n}[1]=[1],
\end{gathered}
$$
induce natural transformations
$$
\rho_{n}^{i}: {\rm K}_{i}(A) \rightarrow {\rm K}_{i}(A ; \mathbb{Z}_n)
$$
$$
\kappa_{m n, m}^{i}: {\rm K}_{i}(A ; \mathbb{Z}_m) \rightarrow {\rm K}_{i}(A ; \mathbb{Z}_{mn}),
$$
$$
\kappa_{n, m n}^{i}: {\rm K}_{i}(A ; \mathbb{Z}_{mn}) \rightarrow {\rm K}_{i}(A ; \mathbb{Z}_n) .
$$
The Bockstein operation
$$
\beta_{n}^{i}: {\rm K}_{i}(A ; \mathbb{Z}_n) \rightarrow {\rm K}_{i+1}(A)
$$
appears in the six-term exact sequence
$$
{\rm K}_{i}(A) \stackrel{\times n}{\longrightarrow} {\rm K}_{i}(A) \stackrel{\rho_{n}^{i}}{\longrightarrow} {\rm K}_{i}(A ; \mathbb{Z} / n) \stackrel{\beta_{n}^{i}}{\longrightarrow} {\rm K}_{i+1}(A) \stackrel{\times n}{\longrightarrow} {\rm K}_{i+1}(A)
$$
induced by the cofibre sequence
$$
A \otimes S C_{0}\left(\mathbb{T}\right) \longrightarrow A \otimes C_{0}\left(W_{n}\right) \stackrel{\beta}{\longrightarrow} A \otimes C_{0}\left(\mathbb{T}\right) \stackrel{n}{\longrightarrow} A \otimes C_{0}\left(\mathbb{T}\right),
$$
where $S C_{0}\left(\mathbb{T}\right)$ is the suspension algebra of $C_{0}\left(\mathbb{T}\right)$.

There is a second six-term exact sequence involving the Bockstein operations. This is induced by a cofibre sequence
$$
A \otimes S C_{0}\left(W_{n}\right) \longrightarrow A \otimes C_{0}\left(W_{m}\right) \longrightarrow A \otimes C_{0}\left(W_{m n}\right) \longrightarrow A \otimes C_{0}\left(W_{n}\right)
$$
and takes the form:
$$
{\rm K}_{i+1}(A ; \mathbb{Z}_n) \stackrel{\beta_{m, n}^{i+1}}{\longrightarrow} {\rm K}_{i}(A ; \mathbb{Z}_m) \stackrel{\kappa_{m n, m}^{i}}{\longrightarrow} {\rm K}_{i}(A ; \mathbb{Z}_{mn}) \stackrel{\kappa_{n, m n}^{i}}{\longrightarrow} {\rm K}_{i}(A ; \mathbb{Z}_n)
$$
where $\beta_{m, n}^{i}=\rho_{m}^{i+1} \circ \beta_{n}^{i}$.

The collection of all the transformations $\rho, \beta, \kappa$ and their compositions is denoted by $\Lambda$. $\Lambda$ can be regarded as the set of morphisms in a category whose objects are the elements of $\mathbb{Z}_2 \times \mathbb{Z}^{+}$. Abusing the terminology $\Lambda$ will be called the category of Bockstein operations. Via the Bockstein operations, $\underline{{\rm K}}(A)$ becomes a $\Lambda$-module. It is natural to consider the group $\operatorname{Hom}_{\Lambda}(\underline{{\rm K}}(A), \underline{{\rm K}}(B))$ consisting of all $\mathbb{Z}_2 \times \mathbb{Z}^{+}$ graded group morphisms which are $\Lambda$-linear, i.e. preserve the action of the category $\Lambda$.

The Kasparov product induces a map
$$
\gamma_{n}^{i}: {\rm K K}(A, B) \rightarrow \operatorname{Hom}\left({\rm K}_{i}(A ; \mathbb{Z} / n), {\rm K}_{i}(B ; \mathbb{Z} / n)\right)
$$
Then $\gamma_{n}=\left(\gamma_{n}^{0}, \gamma_{n}^{1}\right)$ will be a map
$$
\gamma_{n}: {\rm K K}(A, B) \rightarrow \operatorname{Hom}\left({\rm K}_{*}(A ; \mathbb{Z}_n), {\rm K}_{*}(B ; \mathbb{Z}_n)\right).
$$
Note that if $n=0$ then ${\rm K}_{*}(A, \mathbb{Z}_n)={\rm K}_{*}(A)$ and the map $\gamma_{0}$ is the same as the map $\gamma$ from the universal coefficient theorem (UCT) of Rosenberg and Schochet \cite{RS}. We assemble the sequence $\left(\gamma_{n}\right)$ into a map $\Gamma$. Since the Bockstein operations are induced by multiplication with suitable KK elements and since the Kasparov product is associative, we obtain a map
$$
\Gamma: {\rm K K}(A, B) \rightarrow \operatorname{Hom}_{\Lambda}(\underline{{\rm K}}(A), \underline{{\rm K}}(B)) \text {. }
$$
For the sake of simplicity, if $\alpha \in {\rm K K}(A, B)$, then $\Gamma(\alpha)$ will be often denoted by $\alpha_{*}$. 

Assume that $A$ is a separable $C^*$-algebra of stable rank one, the following is a positive cone for total K-theory of $A$ (\cite[Definition 4.6]{DG}):
$$
\underline{\mathrm{K}}(A)_+=\{([e],\mathfrak{u},
\oplus_{p=1}^{\infty}(\mathfrak{s}_p,\mathfrak{s}^p)):\, [e]\in \mathrm{K}_0^+(A),([e],\mathfrak{u},\oplus_{p=1}^{\infty}
(\mathfrak{s}_p,\mathfrak{s}^p))\in \underline{\mathrm{K}}_{I_e}(A)\},
$$
where $\mathfrak{u}\in {\rm K}_1(A)$, $\mathfrak{s}_p\in {\rm K}_0(A;\mathbb{Z}_p)$ and  $\mathfrak{s}^p\in {\rm K}_1(A;\mathbb{Z}_p)$ are equivalent classes,  $I_e$ is the ideal of $A\otimes\mathcal{K}$ generated by $e$ and $\underline{\mathrm{K}}_{I_e}(A)$ is the image of $\underline{\mathrm{K}}({I_e})$ in $\underline{\mathrm{K}}(A)$.
\end{notion}
\begin{theorem}{\rm (\cite[Porposition 4.8--4.9]{DG})}\label{ordertotal}
Suppose that $A$ is of stable rank one and has an approximate unit $(e_n)$ consisting
of projections. Then

(i) $\underline{\mathrm{K}}(A)=\underline{\mathrm{K}}(A)_+-\underline{\mathrm{K}}(A)_+$;

(ii) $\underline{\mathrm{K}}(A)_+\cap\{-\underline{\mathrm{K}}(A)_+\} = \{0\}$, and hence,
$(\underline{\mathrm{K}}(A),\underline{\mathrm{K}}(A)_+)$ is an ordered group;

(iii) For any $x\in \underline{\mathrm{K}}(A)$, there are positive integers $k$, $n$ such that $k[e_n]+x \in \underline{\mathrm{K}}(A)_+$.
\end{theorem}

\begin{proposition}{\rm (\cite[4.10--4.11]{DG})}\label{k-total continuous}
  $\underline{\mathrm{K}}(\cdot)$ and $\underline{\mathrm{K}}(\cdot)_+$ are continuous functors for $C^*$-algebras of stable rank one.
\end{proposition}
\begin{notion}\rm \label{theta and delta}
Let $A$ be a separable $C^*$-algebra of stable rank one and let $I$ and $J$ be ideals of $A\otimes \mathcal{K}$ such that $I\subset J$. Denote  $\theta_{IJ}:\,I\to J$ the natural embedding map and denote
$\delta_{IJ}$ the map ${\rm \underline{ K}}(\theta_{IJ})
: \underline{\mathrm{K}}(I)\to \underline{\mathrm{K}}(J)$, i.e.,
\begin{align*}
   & \delta_{IJ}([e],\mathfrak{u},\oplus_{p=1}^{\infty}(\mathfrak{s}_p,\mathfrak{s}^p)) \\
   =& ({\rm K}_0(\theta_{IJ})([e]) ,{\rm K}_1(\theta_{IJ})(\mathfrak{u}),\oplus_{p=1}^{\infty}
({\rm K}_0(\theta_{IJ};\mathbb{Z}_p)(\mathfrak{s}_p),{\rm K}_1(\theta_{IJ};\mathbb{Z}_p)(\mathfrak{s}^p))),
\end{align*}
where the maps
$${\rm K}_0(\theta_{IJ};\mathbb{Z}_p): {\rm K}_0(I;\mathbb{Z}_p)\rightarrow {\rm K}_0(J;\mathbb{Z}_p),\,\,{\rm K}_1(\theta_{IJ};\mathbb{Z}_p): {\rm K}_1(I;\mathbb{Z}_p)\rightarrow {\rm K}_1(J;\mathbb{Z}_p)$$
are induced by $\theta_{IJ}$.

In this paper, we will always identify
${\rm K}_1(I)\times\bigoplus_{n=1}^{\infty} {\rm K}_* (I; \mathbb{Z}_p)$ with its natural image in
${\rm \underline{K}}(I)=\bigoplus_{p=0}^{\infty} {\rm K}_{*}(I ; \mathbb{Z}_p)$, i.e.,
each $(\mathfrak{u},\oplus_{p=1}^{\infty}(\mathfrak{s}_p,\mathfrak{s}^p))$ is identified with $(0,\mathfrak{u},\oplus_{p=1}^{\infty}(\mathfrak{s}_p,\mathfrak{s}^p))\in {\rm \underline{K}}(I)$.
Then the restriction of $\delta_{IJ}$ also induces a map  
${\rm K}_1(I)\times\bigoplus_{p=1}^{\infty} {\rm K}_* (I; \mathbb{Z}_p)\to{\rm K}_1(J)\times\bigoplus_{p=1}^{\infty} {\rm K}_* (J; \mathbb{Z}_p)$, which is  still denoted by $\delta_{IJ}$.

\end{notion}
\begin{definition}\label{cu total def}\rm
  Let $A$ be a separable $C^*$-algebra of stable rank one. 
  Define
  $$
  {\rm \underline{Cu}}(A)\triangleq \coprod_{I\in {\rm Lat}_f(A)} {\rm Cu}_f(I)\times {\rm K}_1(I)\times\bigoplus_{p=1}^{\infty} {\rm K}_* (I; \mathbb{Z}_p).
  $$


  We equip ${\rm \underline{Cu}}(A)$ with addition and order as follows:
  For any
  $$
  (x,\mathfrak{u},\oplus_{p=1}^{\infty}(\mathfrak{s}_p,\mathfrak{s}^p))\in{\rm Cu}_f(I_x)\times {\rm K}_1(I_x)\times\bigoplus_{p=1}^{\infty} {\rm K}_* (I_x; \mathbb{Z}_p)
  $$
  and
$$
  (y,\mathfrak{v},\oplus_{p=1}^{\infty}(\mathfrak{t}_p,\mathfrak{t}^p))\in{\rm Cu}_f(I_y)\times {\rm K}_1( I_y)\times\bigoplus_{p=1}^{\infty} {\rm K}_* (I_y; \mathbb{Z}_p),
  $$
then
$$
(x,\mathfrak{u},\oplus_{p=1}^{\infty}(\mathfrak{s}_p,\mathfrak{s}^p))
+(y,\mathfrak{v},\oplus_{p=1}^{\infty}(\mathfrak{t}_p,\mathfrak{t}^p))
$$
$$
=(x+y,\delta_{I_xI_{x+y}}(\mathfrak{u},\oplus_{p=1}^{\infty}(\mathfrak{s}_p,\mathfrak{s}^p))
+\delta_{I_yI_{x+y}}(\mathfrak{v},\oplus_{p=1}^{\infty}(\mathfrak{t}_p,\mathfrak{t}^p)))
$$
and
$$
(x,\mathfrak{u},\oplus_{p=1}^{\infty}(\mathfrak{s}_p,\mathfrak{s}^p))
\leq(y,\mathfrak{v},\oplus_{p=1}^{\infty}(\mathfrak{t}_p,\mathfrak{t}^p)),
$$
if
$$\,\,x\leq y\,\,{\rm and}\,\,\delta_{I_xI_y}
(\mathfrak{u},\oplus_{p=1}^{\infty}(\mathfrak{s}_p,\mathfrak{s}^p))=(\mathfrak{v},
\oplus_{p=1}^{\infty}(\mathfrak{t}_p,\mathfrak{t}^p)),
$$
where $\delta_{I_xI_y}$ is the natural map $\underline{\rm K}( I_x)\rightarrow \underline{\rm K}( I_y)$ (see \ref{theta and delta}), $I_x$ and $I_y$ are ideals in $A\otimes\mathcal{K}$ generated by $x$ and $y$, respectively.
\end{definition}

It is easy to see that $({\rm \underline{Cu}}(A),+,\leq)$ is a partially ordered monoid, 
and ${\rm \underline{Cu}}(A)_+\triangleq\{x\in{\rm \underline{Cu}}(A):\,0\leq x \}={\rm Cu}(A)$.
If $A$ is simple, we will have
$${\rm \underline{Cu}}(A)=({\rm {Cu}}(A)\setminus\{0\})\times {\rm K}_1(A)\times\bigoplus_{p=1}^{\infty} {\rm K}_* (A, \mathbb{Z}_p)\cup \{0,0,\bigoplus_{p=1}^{\infty}(0,0)\}.$$

\begin{proposition}
Let $A$ be a separable $C^*$-algebra of stable rank one, the addition and order for ${\rm \underline{Cu}}(A)$ are compatible, i.e., for any $$(x_1,\mathfrak{u}_1,\oplus_{p=1}^{\infty}(\mathfrak{s}_{1,p},\mathfrak{s}^{1,p}))
\leq(y_1,\mathfrak{v}_1,\oplus_{p=1}^{\infty}(\mathfrak{t}_{1,p},\mathfrak{t}^{1,p}))$$
and
$$(x_2,\mathfrak{u}_2,\oplus_{p=1}^{\infty}(\mathfrak{s}_{2,p},\mathfrak{s}^{2,p}))
\leq(y_2,\mathfrak{v}_2,\oplus_{p=1}^{\infty}(\mathfrak{t}_{2,p},\mathfrak{t}^{2,p})),$$
we have
$$(x_1,\mathfrak{u}_1,\oplus_{p=1}^{\infty}(\mathfrak{s}_{1,p},\mathfrak{s}^{1,p}))+(x_2,\mathfrak{u}_2,\oplus_{p=1}^{\infty}(\mathfrak{s}_{2,p},\mathfrak{s}^{2,p}))$$
$$\leq(y_1,\mathfrak{v}_1,\oplus_{p=1}^{\infty}(\mathfrak{t}_{1,p},\mathfrak{t}^{1,p}))+(y_2,\mathfrak{v}_2,\oplus_{p=1}^{\infty}(\mathfrak{t}_{2,p},\mathfrak{t}^{2,p})).$$
\end{proposition}
\begin{proof}
From assumption, we have
$$
x_1\leq y_1,\,\,x_2\leq y_2,
$$
$$
\delta_{I_{x_1}I_{y_1}}(\mathfrak{u}_1,\oplus_{p=1}^{\infty}(\mathfrak{s}_{1,p},\mathfrak{s}^{1,p}))
=(\mathfrak{v}_1,\oplus_{p=1}^{\infty}(\mathfrak{t}_{1,p},\mathfrak{t}^{1,p}))
$$
and
$$
\delta_{I_{x_2}I_{y_2}}(\mathfrak{u}_2,\oplus_{p=1}^{\infty}(\mathfrak{s}_{2,p},\mathfrak{s}^{2,p}))
=(\mathfrak{v}_2,\oplus_{p=1}^{\infty}(\mathfrak{t}_{2,p},\mathfrak{t}^{2,p})).
$$

Since
$$
(x_1,\mathfrak{u}_1,\oplus_{p=1}^{\infty}(\mathfrak{s}_{1,p},\mathfrak{s}^{1,p}))+(x_2,\mathfrak{u}_2,\oplus_{p=1}^{\infty}(\mathfrak{s}_{2,p},\mathfrak{s}^{2,p}))
$$
$$
=(x_1+x_2,\delta_{I_{x_1}I_{x_1+x_2}}(\mathfrak{u}_1,\oplus_{p=1}^{\infty}(\mathfrak{s}_{1,p},\mathfrak{s}^{1,p}))
+\delta_{I_{x_2}I_{x_1+x_2}}(\mathfrak{u}_2,\oplus_{p=1}^{\infty}(\mathfrak{s}_{2,p},\mathfrak{s}^{2,p})))
$$
and
$$
(y_1,\mathfrak{v}_1,\oplus_{p=1}^{\infty}(\mathfrak{t}_{1,p},\mathfrak{t}^{1,p}))+(y_2,\mathfrak{v}_2,\oplus_{p=1}^{\infty}(\mathfrak{t}_{1,p},\mathfrak{t}^{1,p}))
$$
$$
=(y_1+y_2,\delta_{I_{y_1}I_{y_1+y_2}}(\mathfrak{v}_1,\oplus_{p=1}^{\infty}(\mathfrak{t}_{1,p},\mathfrak{t}^{1,p}))
+\delta_{I_{y_2}I_{y_1+y_2}}(\mathfrak{v}_2,\oplus_{p=1}^{\infty}(\mathfrak{t}_{2,p},\mathfrak{t}^{2,p}))).
$$

Note that the following diagram is naturally commutative.
$$
\xymatrixcolsep{3pc}
\xymatrix{
{\,\,\underline{\rm K}(I_{x_1})\,\,} \ar[rd]_-{} \ar[r]^-{}
& {\,\,\underline{\rm K}(I_{y_1})\,\,}  \ar[rd]^-{}
\\
{\,\, \,\,}
& {\,\,\underline{\rm K}(I_{x_1+x_2}) \,\,}  \ar[r]^-{}
& {\,\,\underline{\rm K}(I_{y_1+y_2}) \,\,}
\\
{\,\, \underline{\rm K}(I_{x_2})\,\,} \ar[r]^-{} \ar[ru]^-{}
& {\,\,\underline{\rm K}(I_{y_2}) \,\,}  \ar[ru]^-{}}
$$
Then we have
$$
\delta_{I_{x_1+x_2}I_{y_1+y_2}}(\delta_{I_{x_1}I_{x_1+x_2}}(\mathfrak{u}_1,\oplus_{p=1}^{\infty}(\mathfrak{s}_{1,p},\mathfrak{s}^{1,p}))
+\delta_{I_{x_2}I_{x_1+x_2}}(\mathfrak{u}_2,\oplus_{p=1}^{\infty}(\mathfrak{s}_{2,p},\mathfrak{s}^{2,p})))
$$
$$
=\delta_{I_{y_1}I_{y_1+y_2}}( \delta_{I_{x_1}I_{y_1}}(\mathfrak{v}_1,\oplus_{p=1}^{\infty}(\mathfrak{t}_{1,p},\mathfrak{t}^{1,p})))
+\delta_{I_{y_2}I_{y_1+y_2}}( \delta_{I_{x_1}I_{y_2}}(\mathfrak{v}_2,\oplus_{p=1}^{\infty}(\mathfrak{t}_{2,p},\mathfrak{t}^{2,p})))
$$
\begin{flushleft}
$=\delta_{I_{y_1}I_{y_1+y_2}}((\mathfrak{v}_1,\oplus_{p=1}^{\infty}(\mathfrak{t}_{1,p},\mathfrak{t}^{1,p})))
 +\delta_{I_{y_2}I_{y_1+y_2}}((\mathfrak{v}_2,\oplus_{p=1}^{\infty}(\mathfrak{t}_{2,p},\mathfrak{t}^{2,p})))$
\end{flushleft}
\begin{flushleft}
 $=(\mathfrak{v}_1,\oplus_{p=1}^{\infty}(\mathfrak{t}_{1,p},\mathfrak{t}^{1,p}))+(\mathfrak{v}_2,\oplus_{p=1}^{\infty}(\mathfrak{t}_{2,p},\mathfrak{t}^{2,p}).$
\end{flushleft}

Combining this with
$$
x_1+x_2\leq y_1+y_2\in{\rm {Cu}}(A),
$$
the conclusion is true.

\end{proof}

We now show that $({\rm \underline{Cu}}(A),+,\leq)$ satisfies the axioms in \ref{Cuaxiom}. The proofs are similar to what has been done for the unitary Cuntz semigroup in \cite{L1}.

\begin{lemma}\label{O 0}
Let $A$ be a separable $C^*$-algebra of stable rank one, $x_n\in {\rm Cu}(A)$ be an increasing sequence with supremum $x$. Then for any
$$
\mathfrak{u}\in {\rm K}_1(I_x),\,\, (\mathfrak{s}_p,\mathfrak{s}^p)\in {\rm K}_* (I_x; \mathbb{Z}_p),
$$
there exists $N\in \mathbb{N}$ such that for all $n>N$, there exist
$$
\mathfrak{u}_n\in {\rm K}_1(I_{x_n}),\,\, (\mathfrak{s}_{n,p},\mathfrak{s}^{n,p})\in {\rm K}_* (I_{x_n}; \mathbb{Z}_p)
$$
such that
$$
(x_n,\mathfrak{u}_n,\oplus_{p=1}^{\infty}(\mathfrak{s}_{n,p},\mathfrak{s}^{n,p}))\leq (x,\mathfrak{u},\oplus_{p=1}^{\infty}(\mathfrak{s}_p,\mathfrak{s}^p)).
$$
\end{lemma}
\begin{proof}
  Since $x=\sup_{n}x_n\in {\rm Cu}(A)$, we have
  $
 I_x={\rm C^*-}\lim\limits_{\longrightarrow}I_{x_n}.
  $
  Then by continuity listed in Proposition \ref{k-total continuous}, we have
  $$
  \underline{\rm K}(I_x)=\lim\limits_{\longrightarrow}(\underline{\rm K}(I_{x_n}),\delta_{I_{x_n}I_{x_m}}),
  $$
 where $\underline{\rm K}(I_x)$ is an algebraic limit.

  Hence, there exists $N$ large enough such that for all $n\geq N$, there exist
  $$\mathfrak{u}_n\in {\rm K}_1(I_{x_n})\quad {\rm and}\quad (\mathfrak{s}_{n,p},\mathfrak{s}^{n,p})\in {\rm K}_* (I_{x_n}; \mathbb{Z}_p)$$
  with
   $$\delta_{I_{x_n}I_x}(\mathfrak{u}_n,\oplus_{p=1}^{\infty}(\mathfrak{s}_{n,p},\mathfrak{s}^{n,p}))= (\mathfrak{u},\oplus_{p=1}^{\infty}(\mathfrak{s}_p,\mathfrak{s}^p)).$$
  Now we get
 $$
(x_n,\mathfrak{u}_n,\oplus_{p=1}^{\infty}(\mathfrak{s}_{n,p},\mathfrak{s}^{n,p}))\leq (x,\mathfrak{u},\oplus_{p=1}^{\infty}(\mathfrak{s}_p,\mathfrak{s}^p)).
$$

\end{proof}

\begin{corollary}\label{O 1}
Let $A$ be a separable $C^*$-algebra of stable rank one. Then any increasing sequence $(x_n,\mathfrak{u}_n,\oplus_{p=1}^{\infty}(\mathfrak{s}_{n,p},\mathfrak{s}^{n,p}))$ in ${\rm \underline{Cu}}(A)$ has a supremum $(x,\mathfrak{u},\oplus_{p=1}^{\infty}(\mathfrak{s}_p,\mathfrak{s}^p))$.
In particular, $x=\sup_n x_n$, and for any $n\in \mathbb{N}$,
$$
\delta_{I_{x_n}I_x}(\mathfrak{u}_n,\oplus_{p=1}^{\infty}(\mathfrak{s}_{n,p},\mathfrak{s}^{n,p}))= (\mathfrak{u},\oplus_{p=1}^{\infty}(\mathfrak{s}_p,\mathfrak{s}^p)).
$$
\end{corollary}
\begin{proof}
Since ${\rm Cu}(A)$ satisfies (O1), there exists $x\in {\rm Cu}(A)$ such that $x=\sup_n x_n$, and   $$
 I_x={\rm C^*-}\lim_{\longrightarrow}I_{x_n}.
  $$ Then
$$
\underline{\rm K}(I_x)=\lim_{\longrightarrow}(\underline{\rm K}(I_{x_n}),\delta_{I_{x_n}I_{x_m}}).
$$

Note that
$$
\delta_{I_{x_n}I_{x_n+1}}(\mathfrak{u}_n,\oplus_{p=1}^{\infty}
(\mathfrak{s}_{n,p},\mathfrak{s}^{n,p}))= (\mathfrak{u}_{n+1},\oplus_{p=1}^{\infty}(\mathfrak{s}_{n+1,p},\mathfrak{s}^{n+1,p})).
$$
We may choose $(0,\mathfrak{u},\oplus_{p=1}^{\infty}
(\mathfrak{s}_p,\mathfrak{s}^p))\in \underline{\rm K}(I_x)$ be the limit
$$
(0,\mathfrak{u},\oplus_{p=1}^{\infty}(\mathfrak{s}_p,\mathfrak{s}^p)):=\lim_{\longrightarrow}
\delta_{I_{x_n}I_{x}}(0,\mathfrak{u}_n,\oplus_{p=1}^{\infty}(\mathfrak{s}_{n,p},\mathfrak{s}^{n,p}))
$$
Then for any $n\in\mathbb{N}$,
$$(x_n,\mathfrak{u}_n,\oplus_{p=1}^{\infty}(\mathfrak{s}_{n,p},\mathfrak{s}^{n,p}))
\leq(x,\mathfrak{u},\oplus_{p=1}^{\infty}(\mathfrak{s}_p,\mathfrak{s}^p)).
$$

Now we check that $(x,\mathfrak{u},\oplus_{p=1}^{\infty}(\mathfrak{s}_p,\mathfrak{s}^p))$ is in fact the supremum of the sequence, choose any $(y,\mathfrak{v},\oplus_{p=1}^{\infty}(\mathfrak{t}_p,\mathfrak{t}^p))\geq (x_n,\mathfrak{u}_n,\oplus_{p=1}^{\infty}(\mathfrak{s}_{n,p},\mathfrak{s}^{n,p}))$ for all $n\in \mathbb{N}$,
then $x=\sup_{n\in\mathbb{N}} x_n \leq y$ and
$$
\delta_{I_{x_n}I_{y}}(\mathfrak{u}_n,\oplus_{p=1}^{\infty}(\mathfrak{s}_{n,p},\mathfrak{s}^{n,p}))= (\mathfrak{v},\oplus_{p=1}^{\infty}(\mathfrak{t}_p,\mathfrak{t}^p)).
$$

It is obvious that
$$
\delta_{I_{x}I_{y}}(\mathfrak{u},\oplus_{p=1}^{\infty}(\mathfrak{s}_p,\mathfrak{s}^p))= (\mathfrak{v},\oplus_{p=1}^{\infty}(\mathfrak{t}_p,\mathfrak{t}^p)).
$$
Hence,
$$
(x,\mathfrak{u},\oplus_{p=1}^{\infty}(\mathfrak{s}_p,\mathfrak{s}^p))\leq (y,\mathfrak{v},\oplus_{p=1}^{\infty}(\mathfrak{t}_p,\mathfrak{t}^p)).
$$
\end{proof}

\begin{theorem}\label{jin bijiao}
Let $A$ be a separable $C^*$-algebra of stable rank one and let $$(x,\mathfrak{u},\oplus_{p=1}^{\infty}(\mathfrak{s}_p,\mathfrak{s}^p)), (y,\mathfrak{v},\oplus_{p=1}^{\infty}(\mathfrak{t}_p,\mathfrak{t}^p))\in \underline{\rm Cu}(A).$$
 Then
$$(x,\mathfrak{u},\oplus_{p=1}^{\infty}(\mathfrak{s}_p,\mathfrak{s}^p))\ll (y,\mathfrak{v},\oplus_{p=1}^{\infty}(\mathfrak{t}_p,\mathfrak{t}^p)).$$
 if and only if $x\ll y$ in ${\rm Cu}(A)$ and $$\delta_{I_xI_y}(\mathfrak{u},\oplus_{p=1}^{\infty}
 (\mathfrak{s}_p,\mathfrak{s}^p))=(\mathfrak{v},\oplus_{p=1}^{\infty}(\mathfrak{t}_p,\mathfrak{t}^p))
.$$
\end{theorem}
\begin{proof}
  Suppose that $(x,\mathfrak{u},\oplus_{p=1}^{\infty}
  (\mathfrak{s}_p,\mathfrak{s}^p))\ll (y,\mathfrak{v},\oplus_{p=1}^{\infty}(\mathfrak{t}_p,\mathfrak{t}^p))$, then
 $$
 \delta_{I_xI_y}(\mathfrak{u},\oplus_{p=1}^{\infty}
 (\mathfrak{s}_p,\mathfrak{s}^p))=(\mathfrak{v},\oplus_{p=1}^{\infty}(\mathfrak{t}_p,\mathfrak{t}^p)),
$$
we only need to prove $x\ll y$.

Let $(y_n)_n$ be an increasing sequence in ${\rm Cu}(A)$ with supremum $y$. Then
$$
\underline{\rm K}(I_y)=\lim_{\longrightarrow}(\underline{\rm K}(I_{y_n}),\delta_{I_{y_n}I_{y_m}}).
$$
There exists $n_0\in \mathbb{N}$ such that for any $n\geq n_0$,
   $$
  \delta_{I_{y_n}I_{y}}(\mathfrak{v}_n,\oplus_{p=1}^{\infty}(\mathfrak{t}_{n,p},\mathfrak{t}^{n,p})=  (\mathfrak{v},\oplus_{p=1}^{\infty}(\mathfrak{t}_p,\mathfrak{t}^p)),
  $$
where $(\mathfrak{v}_n,\oplus_{p=1}^{\infty}(\mathfrak{t}_{n,p},\mathfrak{t}^{n,p}))$ is identified with $(0,\mathfrak{v}_n,\oplus_{p=1}^{\infty}(\mathfrak{t}_{n,p},\mathfrak{t}^{n,p}))\in \underline{\rm K}(I_{y_n}).$

Then $((y_n,\mathfrak{v}_n,\oplus_{p=1}^{\infty}(\mathfrak{t}_{n,p},\mathfrak{t}^{n,p})))_{n\geq n_0}$ is an increasing sequence in  ${\rm \underline{Cu}}(A)$ with supremum $(y,\mathfrak{v},\oplus_{p=1}^{\infty}(\mathfrak{t}_p,\mathfrak{t}^p)).$ From assumption, there exists $N\in \mathbb{N}$ ($N\geq n_0$) such that for any $n\geq N$,
$$
(x,\mathfrak{u},\oplus_{p=1}^{\infty}(\mathfrak{s}_p,\mathfrak{s}^p))\leq (y_n,\mathfrak{v}_n,\oplus_{p=1}^{\infty}(\mathfrak{t}_{n,p},\mathfrak{t}^{n,p})).
$$
Then we have $x\leq y_n$ for all $n\geq N$, that is, $x\ll y.$

Conversely, suppose that $((y_n,\mathfrak{v}_n,\oplus_{p=1}^{\infty}(\mathfrak{t}_{n,p},\mathfrak{t}^{n,p})))_n$ is an increasing sequence in ${\rm \underline{Cu}}(A)$ with supremum $(y,\mathfrak{v},\oplus_{p=1}^{\infty}(\mathfrak{t}_p,\mathfrak{t}^p)).$  Then $(y_n)_n$ is an increasing sequence in
$ {\rm {Cu}}(A)$ with supremum ${y}.$ Then there exists $N$ large enough such that for any $n> N$,
$$
x\leq y_n\leq y,\quad I_x\subset I_{y_n}\subset I_y
$$
and
$$
\delta_{I_{y_n}I_y}(\mathfrak{v}_n,\oplus_{p=1}^{\infty}(\mathfrak{t}_{n,p},\mathfrak{t}^{n,p}))= (\mathfrak{v},\oplus_{p=1}^{\infty}(\mathfrak{t}_p,\mathfrak{t}^p)).
$$
where $(\mathfrak{v}_n,\oplus_{p=1}^{\infty}(\mathfrak{t}_{n,p},\mathfrak{t}^{n,p})$ is identified as $(0,\mathfrak{v}_n,\oplus_{p=1}^{\infty}(\mathfrak{t}_{n,p},\mathfrak{t}^{n,p})\in \underline{\rm K}(I_{y_n}).$

Now we have
$$\delta_ {I_xI_{y_{n}}}(\mathfrak{u},\oplus_{p=1}^{\infty}
(\mathfrak{s}_p,\mathfrak{s}^p))=(\mathfrak{v}_{n},\oplus_{p=1}^{\infty}(\mathfrak{t}_{n,p},\mathfrak{t}^{n,p})),
$$
this means that
$$
(x,\mathfrak{u},\oplus_{p=1}^{\infty}(\mathfrak{s}_p,\mathfrak{s}^p))\leq(y_n,\mathfrak{v}_n,\oplus_{p=1}^{\infty}(\mathfrak{t}_{n,p},\mathfrak{t}^{n,p}))\quad {\rm for\,\, all\,\,} n\geq N.
$$
Hence,
$$
(x,\mathfrak{u},\oplus_{p=1}^{\infty}(\mathfrak{s}_p,\mathfrak{s}^p))\ll(y,\mathfrak{v},\oplus_{p=1}^{\infty}(\mathfrak{t}_p,\mathfrak{t}^p)).
$$
\end{proof}
\begin{corollary}\label{O 2}
  Let $A$ be a separable $C^*$-algebra of stable rank one, then $(x,\mathfrak{u},\oplus_{p=1}^{\infty}(\mathfrak{s}_p,\mathfrak{s}^p))$  is compact in ${\rm \underline{Cu}}(A)$ if and only if $x$ is compact in ${\rm Cu}(A)$.
\end{corollary}
\begin{corollary}\label{alg total cu}
  Let $A$ be a separable $C^*$-algebra of stable rank one, then $A$ is of real rank zero if and only if ${\rm \underline{Cu}}(A)$ is algebraic.
\end{corollary}
Then we obtain the following theorem.
\begin{theorem}
  Let $A$ be a separable $C^*$-algebra of stable rank one, then $({\rm \underline{Cu}}(A),+,\leq)$ satisfies axioms {\rm (O1),(O2),(O3),} and {\rm (O4)}.
\end{theorem}
\begin{proof}

  (O1): It is exactly Corollary {\ref{O 1}}.

  (O2): $(x,\mathfrak{v},\oplus_{p=1}^{\infty}(\mathfrak{s}_p,\mathfrak{s}^p))\in {\rm \underline{Cu}}(A),$ from the (O2) of ${\rm Cu}(A)$, we have a $\ll$-increasing sequence $(x_n)_n$ in  ${\rm Cu}(A)$. Then by
  Lemma  {\ref{O 0}} and Theorem \ref{jin bijiao}, we achieve (O2) of $({\rm \underline{Cu}}(A), +, \leq)$.

  (O3): Let $$
(x_1,\mathfrak{u}_1,\oplus_{p=1}^{\infty}(\mathfrak{s}_{1,p},
\mathfrak{s}^{1,p}))\ll(y_1,\mathfrak{v}_1,\oplus_{p=1}^{\infty}(\mathfrak{t}_{1,p},\mathfrak{t}^{1,p}))
$$ and $$
(x_2,\mathfrak{u}_2,\oplus_{p=1}^{\infty}
(\mathfrak{s}_{2,p},\mathfrak{s}^{2,p}))\ll
(y_2,\mathfrak{v}_2,\oplus_{p=1}^{\infty}(\mathfrak{t}_{2,p},\mathfrak{t}^{2,p})).
$$ We have
$$
(x_1,\mathfrak{u}_1,\oplus_{p=1}^{\infty}
(\mathfrak{s}_{1,p},\mathfrak{s}^{1,p}))+
(x_2,\mathfrak{u}_2,\oplus_{p=1}^{\infty}(\mathfrak{s}_{2,p},\mathfrak{s}^{2,p}))
$$
$$\leq
(y_1,\mathfrak{v}_1,\oplus_{p=1}^{\infty}(\mathfrak{t}_{1,p},\mathfrak{t}^{1,p}))
+(y_2,\mathfrak{v}_2,\oplus_{p=1}^{\infty}(\mathfrak{t}_{2,p},\mathfrak{t}^{2,p})).
$$
From (O3) of Cu$(A)$, we also have $x_1+x_2\ll y_1+y_2$, therefore, by Theorem \ref{jin bijiao}, we get (O3) for $({\rm \underline{Cu}}(A), +,\leq)$.

 (O4): Let $((x_n,\mathfrak{u}_n,\oplus_{p=1}^{\infty}(\mathfrak{s}_{n,p},\mathfrak{s}^{n,p})))_n$
and $((y_n,\mathfrak{v}_n,\oplus_{p=1}^{\infty}(\mathfrak{t}_{n,p},\mathfrak{t}^{n,p})))_n$ be two increasing sequences in ${\rm \underline{Cu}}(A)$. Let
$$(x,\mathfrak{u},\oplus_{p=1}^{\infty}(\mathfrak{s}_p,\mathfrak{s}^p))=\sup\limits_{n\in\mathbb{N}}
(x_n,\mathfrak{u}_n,\oplus_{p=1}^{\infty}(\mathfrak{s}_{n,p},\mathfrak{s}^{n,p}))$$
and
$$(y,\mathfrak{v},\oplus_{p=1}^{\infty}(\mathfrak{t}_{p},\mathfrak{t}^{p}))
=\sup_{n\in\mathbb{N}}(y_n,\mathfrak{v}_n,\oplus_{p=1}^{\infty}(\mathfrak{t}_{n,p},\mathfrak{t}^{n,p})).
$$
From (O4) of Cu$(A)$, we have $\sup_{n\in\mathbb{N}}\{x_n+y_n\}=x+y$.
By the compatibility  of order $\leq$ and addition, we know
$$
((x_n,\mathfrak{u}_n,\oplus_{p=1}^{\infty}(\mathfrak{s}_{n,p},\mathfrak{s}^{n,p}))+
(y_n,\mathfrak{v}_n,\oplus_{p=1}^{\infty}(\mathfrak{t}_{n,p},\mathfrak{t}^{n,p})))
$$
is also an increasing sequence.
Note that
$$
\delta_{I_{x_n}I_{x}}(\mathfrak{u}_n,\oplus_{p=1}^{\infty}
(\mathfrak{s}_{n,p},\mathfrak{s}^{n,p}))=
(\mathfrak{u},\oplus_{p=1}^{\infty}(\mathfrak{s}_p,\mathfrak{s}^p))
$$
and
$$
\delta_{I_{y_n}I_{y}}(\mathfrak{v}_n,\oplus_{p=1}^{\infty}(\mathfrak{t}_{n,p},\mathfrak{t}^{n,p}))=
(\mathfrak{v},\oplus_{p=1}^{\infty}(\mathfrak{t}_{p},\mathfrak{t}^{p})).
$$
This implies
$$
\delta_{I_{x_n+y_n}I_{x+y}}(
\delta_{I_{x_n}I_{x_n+y_n}}(\mathfrak{u}_n,\oplus_{p=1}^{\infty}(\mathfrak{s}_{n,p},\mathfrak{s}^{n,p}))
+
\delta_{I_{y_n}I_{x_n+y_n}}(\mathfrak{v}_n,\oplus_{p=1}^{\infty}(\mathfrak{t}_{n,p},\mathfrak{t}^{n,p})))
$$$$=
\delta_{I_{x}I_{x+y}}(\mathfrak{u},\oplus_{p=1}^{\infty}(\mathfrak{s}_p,\mathfrak{s}^p))
+\delta_{I_{y}I_{x+y}}(\mathfrak{v},\oplus_{p=1}^{\infty}(\mathfrak{t}_{p},\mathfrak{t}^{p})).
$$
Using Corollary \ref{O 1}, we get (O4) for $({\rm \underline{Cu}}(A),+, \leq)$.

\end{proof}
\begin{definition}\rm
Define the total Cuntz category ${\rm \underline{Cu}}$ as follows:
$$
{\rm Ob}({\rm \underline{Cu}})=\{\,S\in {\rm Cu}^\sim\,|\,{\rm Gr}(S_c)\,\,{\rm is\,\, a\,\, \Lambda-module}\,\};
$$
let $X,Y\in {\rm \underline{Cu}}$, we say the map $\phi:\,X\to Y$ is a ${\rm \underline{Cu}}$-morphism if $\phi$ satisfies the following two conditions:

(1) $\phi$ is a ${\rm {Cu}}$-morphism, i.e., $\phi$ is ${\rm Mon}_{\leq}$-morphism and preserves suprema of increasing sequences and the compact containment relation.

(2) The induced Grothendieck map ${\rm Gr}(\phi_c): {\rm Gr}(X_c) \rightarrow {\rm Gr}(Y_c)$ is $\Lambda$-linear, i.e, ${\rm Gr}(\phi_c)$ is a morphism of the category $\Lambda$.

\end{definition}


\begin{remark}\label{mor remark}
Let $A$ and $B$ be $C^*$-algebras and let $\psi:\, A\to B$ be a $*$-homomorphism.
We still denote $\psi\otimes {\rm id}_\mathcal{K}$ by $\psi$. Let $x\in {\rm Cu}(A)$ and $y\in {\rm Cu}(B)$, suppose that $\psi(I_x)\subset I_y$, where $I_x$ and $I_y$ are ideals of $A\otimes \mathcal{K}$ and $B\otimes \mathcal{K}$ generated by $x$ and $y$, respectively.
Denote the restriction map $\psi|_{I_x\to I_y}:\,I_x\to I_y$ by $\psi_{I_x I_y}$.

Then $\psi$ induces the two morphisms:
$$
{\rm \underline{K}}(\psi):\,{\rm \underline{K}}(A)\to{\rm \underline{K}}(B)
\quad{\rm and}
\quad {\rm \underline{K}}(\psi_{I_x I_y}):\,{\rm \underline{K}}(I_x)\to{\rm \underline{K}}(I_y).$$

In general, ${\rm \underline{K}}(\psi_{I_x I_y})$ is not the restriction of ${\rm \underline{K}}(\psi)$, as ${\rm \underline{K}}(I_x)$ is not a subgroup of ${\rm \underline{K}}(A)$.

In particular, we define a map ${\rm \underline{Cu}}(\psi):\,{\rm \underline{Cu}}(A)\to {\rm \underline{Cu}}(B)$ by
$$
{\rm \underline{Cu}}(\psi)(x,\mathfrak{u},\oplus_{p=1}^{\infty}(\mathfrak{s}_p,\mathfrak{s}^p))=({\rm {Cu}}(\psi)(x),{\rm \underline{K}}(\psi_{I_x I_{{\rm {Cu}}(\psi)(x)}})(\mathfrak{u},\oplus_{p=1}^{\infty}(\mathfrak{s}_p,\mathfrak{s}^p))).
$$
(Note that $(\mathfrak{u},\oplus_{p=1}^{\infty}(\mathfrak{s}_p,\mathfrak{s}^p))$ is identified with $(0,\mathfrak{u},\oplus_{p=1}^{\infty}(\mathfrak{s}_p,\mathfrak{s}^p))\in {\rm \underline{K}}(I_x)$.)

It is easy to check that
$${\rm \underline{Cu}}(\psi)(x,\mathfrak{u},\oplus_{p=1}^{\infty}(\mathfrak{s}_p,\mathfrak{s}^p))\leq (y,\mathfrak{v},\oplus_{p=1}^{\infty}(\mathfrak{t}_p,\mathfrak{t}^p))$$
if and only if
$
{\rm {Cu}}(\psi)(x)\leq y\in {\rm {Cu}}(B)$ {\rm and}
$$
{\rm \underline{K}}(\psi_{I_x I_y})(0,\mathfrak{u},\oplus_{p=1}^{\infty}(\mathfrak{s}_p,\mathfrak{s}^p))=
(0,\mathfrak{v},\oplus_{p=1}^{\infty}(\mathfrak{t}_p,\mathfrak{t}^p))\in {\rm \underline{K}}(I_y).
$$

\end{remark}


\begin{theorem}\label{Cu total thm}
  The assignment
    \begin{eqnarray*}
   {\rm \underline{Cu}}:\, C_{sr1}^* & \rightarrow & {\rm \underline{Cu}}  \\
    A &\mapsto & {\rm \underline{Cu}}(A) \\
    \phi &\mapsto & {\rm \underline{Cu}}(\phi)
  \end{eqnarray*}
 from the category of unital, separable $C^*$-algebras of stable rank one to the category   ${\rm \underline{Cu}}$ is a functor.

 In particular, we have
 $$
 ({\rm Gr}({\rm \underline{Cu}}(A)_c),\rho({\rm \underline{Cu}}(A)_c))\cong({\rm \underline{K}}(A),{\rm \underline{K}}(A)_+),
 $$
 where $\rho:\, S_c\to {\rm Gr}(S_c)$ is the natural map ($\rho(x)=[(x,0)])$.
\end{theorem}

\begin{proof}
The key point is to prove
  $
  {\rm Gr}({\rm \underline{Cu}}(A)_c)\cong{\rm \underline{K}}(A)$ (canonically), where ${\rm Gr}(\cdot)$ is the Grothendick group procedure.

  Since $A$ has stable rank one, any compact element in ${\rm Cu}(A)$ can be raised by projections \cite[Theorem 5.8]{BC}, then by Corollary \ref{O 2}, every compact element in ${\rm {Cu}}(A)$ has the form $([e],u,\oplus_{p=1}^{\infty}(\mathfrak{s}_p,\mathfrak{s}^p))$, where $e$ is a projection in $A\otimes \mathcal{K}$.

  Define
  \begin{eqnarray*}
            \alpha: {\rm \underline{Cu}}(A)_c &\rightarrow & {\rm \underline{K}}(A)_+ \\
            ([e],\mathfrak{u},\oplus_{p=1}^{\infty}(\mathfrak{s}_p,\mathfrak{s}^p)) &\mapsto & ([e],\delta_{I_eA}(\mathfrak{u},\oplus_{p=1}^{\infty}(\mathfrak{s}_p,\mathfrak{s}^p))).
\end{eqnarray*}
From the definition of $\underline{\rm K}(A)_+$ we listed in \ref{def k-total}, $\alpha$ is a surjective monoid morphism.
  ($\alpha$ is just surjective, not necessarily injective. The $C^*$-algebra $E$ in  Example \ref{example sur not in} is such an algebra.)

  Note that ${\rm {K}}_0^+(A)$ is a sub-cone of $ {\rm \underline{K}}(A)_+$, it induces an order on ${\rm \underline{K}}(A)_+$, that is, for  $ (f,\overline{f}),(g,\overline{g})\in{\rm \underline{K}}(A)_+$ ($f,g\in {\rm K}_0^+(A)$ and $\overline{f},\overline{g}\in {\rm K}_1(A)\times\oplus_{p=1}^{\infty}{\rm K}_*(A;\mathbb{Z}_p)$), we say
$
(f,\overline{f})\leq(g,\overline{g}),
$
if
$
 f\leq g\,\,{\rm and} \,\,\overline{f}=\overline{g}.
$

Then $\alpha$ is an ordered morphism.

  Suppose that
  $$
  \alpha([e],\mathfrak{u},\oplus_{p=1}^{\infty}(\mathfrak{s}_p,\mathfrak{s}^p))
  =\alpha([q],\mathfrak{v},\oplus_{p=1}^{\infty}(\mathfrak{t}_p,\mathfrak{t}^p)),
  $$
  which is
  $$
  ([e],\delta_{I_eA}(\mathfrak{u},\oplus_{p=1}^{\infty}( \mathfrak{s}_p,\mathfrak{s}^p)))=
  ([q],\delta_{I_qA}(\mathfrak{v},\oplus_{p=1}^{\infty}( \mathfrak{t}_p,\mathfrak{t}^p))).
  $$
  Then for $([1_A],0,\oplus_{p=1}^{\infty}(0,0))\in  {\rm \underline{Cu}}(A)_c$,
  in ${\rm \underline{Cu}}(A)_c$, we have
  $$
  ([1_A]+[e],\delta_{I_eA}(\mathfrak{u},\oplus_{p=1}^{\infty}(\mathfrak{s}_p,\mathfrak{s}^p)))
  $$$$=
  ([1_A]+[q],\delta_{I_qA}(\mathfrak{v},\oplus_{p=1}^{\infty}(\mathfrak{t}_p,\mathfrak{t}^p))),
  $$
  which is
  $$
  ([1_A],0,\oplus_{p=1}^{\infty}(0,0))+([e],\mathfrak{u},
  \oplus_{p=1}^{\infty}(\mathfrak{s}_p,\mathfrak{s}^p))
  $$$$=
  ([1_A],0,\oplus_{p=1}^{\infty}(0,0))+([q],\mathfrak{v},
  \oplus_{p=1}^{\infty}(\mathfrak{t}_p,\mathfrak{t}^p)).
  $$

  This means that for any $x,y\in {\rm \underline{Cu}}(A)_c$, if $\alpha(x)=\alpha(y)$, under the Grothendieck construction, the difference between $x$ and $y$ vanishes.

  Consider the natural map $\rho:\, \mathrm{\underline {Cu}}(A)_c\to {\rm Gr}(\mathrm{\underline {Cu}}(A)_c)$, if $\rho(x)=\rho(y)$, then there exists $z\in \mathrm{\underline {Cu}}(A)_c$ such that
$x+z=y+z$, hence, $$\alpha(x)+\alpha(z)=\alpha(y)+\alpha(z),$$ by \ref{ordertotal}, then $\alpha(x)=\alpha(y)$.

  Then we have
  $$
  {\rm Gr}({\rm \underline{Cu}}(A)_c )\cong{\rm Gr}(\alpha({\rm \underline{Cu}}(A)_c))= {\rm Gr}({\rm \underline{K}}(A)_+)={\rm \underline{K}}(A)
  $$
 and
 $$
 ({\rm Gr}({\rm \underline{Cu}}(A)_c),\rho({\rm \underline{Cu}}(A)_c))\cong({\rm \underline{K}}(A),{\rm \underline{K}}(A)_+).
 $$

  Since  ${\rm \underline{K}}(A)$ is a $\Lambda$-module, then ${\rm Gr}({\rm \underline{Cu}}(A)_c)$ is also a $\Lambda$-module, now we have ${\rm \underline{Cu}}(A)\in {\rm \underline{Cu}}$. Suppose that $\phi:A\rightarrow B$ is a $*$-homomorphism, it is obvious that the induced map
  $
  {\rm \underline{Cu}}(\phi): {\rm \underline{Cu}}(A)\rightarrow  {\rm \underline{Cu}}(B)
  $ is a ${\rm {Cu}}^\sim$-morphism.

 For any $A\in C^*_{sr1}$, write the canonical isomorphism induced by $\alpha$ as
  $$
 \alpha_A^*:\,({\rm Gr}({\rm \underline{Cu}}(A)_c),\rho({\rm \underline{Cu}}(A)_c))\to({\rm \underline{K}}(A),{\rm \underline{K}}(A)_+).
 $$
  We identify ${\rm Gr}({\rm \underline{Cu}}(A)_c )$ and ${\rm Gr}({\rm \underline{Cu}}(B)_c )$ with
  ${\rm \underline{K}}(A)$ and ${\rm \underline{K}}(B)$ through $\alpha_A^*$ and $\alpha_B^*$, respectively. Then the Grothendieck map
  $$
  {\rm Gr}(\phi_c): {\rm \underline{K}}(A)\rightarrow {\rm \underline{K}}(B)
  $$
  is exactly the map ${\rm \underline{K}}(\phi)$, and hence, is a $\Lambda$-linear map.

  Then  ${\rm \underline{Cu}}(\phi)$ is a ${\rm \underline{Cu}}^\sim$-morphism.




\end{proof}

Now we give another picture of the total Cuntz semigroup, which seems
a natural construction like $\lesssim_1$ given in \cite[3.1]{L1}.
\begin{notion}\label{cutnz new}\rm
(\cite{Cu1}) For each $C^{*}$-algebra $A$, Cuntz defined the $C^{*}$-algebra $Q A=A * A$ as the free product of $A$ by itself. $Q A$ is generated as a $C^{*}$-algebra by $\{\iota(a), \bar{\iota}(a) \mid a \in A\}$ with respect to the largest $C^*$-norm, where $\iota$ and $\bar{\iota}$ denote the inclusions of two copies of $A$ into the free product. He also defined $q A$ as the ideal of $Q A$ generated by the differences $\{\iota(a)-\bar{\iota}(a) \mid a \in A\}$ and showed that for any  $C^*$-algebras $A, B$,
$$
{\rm  KK}(A,B)=[qA,B\otimes \mathcal{K}],
$$
where $[qA,B\otimes \mathcal{K}]$ consists of all the homotopy classes of homomorphisms from $qA$ to $B\otimes \mathcal{K}$.

Let $\psi:\,A\to B$ be a homomorphism, denote $SA$  the suspension algebra of $A$, i.e., $A\otimes C_0(0,1)$ and set $S\psi=\psi\otimes {\rm id}_{C_0(0,1)}.$ Then we have
 \begin{eqnarray*}
   {\rm K}_*(A;\mathbb{Z}_p) & = & {\rm K}_0(A;\mathbb{Z}_p)\oplus {\rm K}_1(A;\mathbb{Z}_p)  \\
     &= & {\rm  KK}(\mathbb{C},A\otimes  C_0(W_p))\oplus {\rm  KK}(\mathbb{C},SA\otimes  C_0(W_p))\\
     &= & [q\mathbb{C},A\otimes  C_0(W_p)\otimes\mathcal{K}]\oplus [q\mathbb{C},SA\otimes  C_0(W_p)\otimes\mathcal{K}].
  \end{eqnarray*}
\end{notion}
\begin{definition}\rm

Let $A$ be a separable $C^*$-algebras of stable rank one.
Let $a,b\in A\otimes\mathcal{K}$, let $u,v$ be unitary elements of $I_a^\sim$ and $I_b^\sim$, respectively and
let
$$s_p:\,q\mathbb{C}\to I_a\otimes C_0(W_p),\quad
t_p:\,q\mathbb{C}\to I_b\otimes C_0(W_p)$$
and
$$s^p:\,q\mathbb{C}\to SI_a\otimes C_0(W_p),\quad t^p:\,{\mathbb{I}}_p\to SI_b\otimes C_0(W_p)$$ be homomorphisms for $p=1,2,\cdots$.

We say $(a,u,\oplus_{p=1}^{\infty}(s_p,s^p))\lesssim_T(b,v,\oplus_{p=1}^{\infty}(t_p,t^p))$,
if
$$
a\lesssim_{Cu} b,\quad [\theta_{I_aI_b}^\sim(u)]=[v]\in {\rm K}_1(I_b),
$$
$$
[(\theta_{I_aI_b}\otimes {\rm id}_{C_0(W_p)})\circ s_p)]=[t_p]\in {\rm K}_0(I_b,\mathbb{Z}_p)$$and
$$
[(S\theta_{I_aI_b}\otimes {\rm id}_{C_0(W_p)})\circ s^p]=[t^p]\in {\rm K}_1(I_b,\mathbb{Z}_p),\quad p=1,2,\cdots,
$$
where $\theta_{I_aI_b}$ (see \ref{theta and delta}) is the natural embedding map from $I_a$ to $I_b$. ($\theta_{I_aI_b}$ induces the map $\delta_{I_aI_b}$, i.e.,
${\rm \underline{K}}(\theta_{I_aI_b})=\delta_{I_aI_b}$.)
\end{definition}
It is obvious that the relation $\lesssim_T$ is reflexive and transitive.
\begin{definition}\rm
Let $A$ be a separable $C^*$-algebra of stable rank one.
Denote $\mathfrak{T}(A)$ the set of the all tuples $(a,u,\oplus_{p=1}^{\infty}(s_p,s^p))$, where $a\in (A\otimes\mathcal{K})_+$, $u\in \mathcal{U}(I_a^\sim)$ and $s_p\in {\rm Hom}(q\mathbb{C}, I_a\otimes C_0(W_p))$, $s^p\in {\rm Hom}(q\mathbb{C}, SI_a\otimes C_0(W_p))$.

By antisymmetrizing the $\lesssim_T$ relation, we define an equivalent relation $\sim_T$ on $\mathfrak{T}(A)$ called the total Cuntz equivalence, and denote $[(a,u,\oplus_{p=1}^{\infty}(s_p,s^p))]$ the equivalent class of $(a,u,\oplus_{p=1}^{\infty}(s_p,s^p))$. We construct a semigroup of $A$ as follows:
$$
\mathcal{T}(A):=\{[(a,u,\oplus_{p=1}^{\infty}(s_p,s^p))]\mid (a,u,\oplus_{p=1}^{\infty}(s_p,s^p))\in \mathfrak{T}(A)\}/\sim_T.
$$

For any two $[(a,u,\oplus_{p=1}^{\infty}(s_p,s^p))],[(b,v,\oplus_{p=1}^{\infty}(t_p,t^p))]\in \mathcal{T}(A)$, we say $$[(a,u,\oplus_{p=1}^{\infty}(s_p,s^p))]\leq [(b,v,\oplus_{p=1}^{\infty}(t_p,t^p))],$$ if
$$(a,u,\oplus_{p=1}^{\infty}(s_p,s^p))\lesssim_T (b,v,\oplus_{p=1}^{\infty}(t_p,t^p)).$$

The following addition is well-defined and compatible with the relation $\leq$ on $\mathcal{T}(A)$:
$$
[(a,u,\oplus_{p=1}^{\infty}(s_p,s^p))]+[(b,v,\oplus_{p=1}^{\infty}(t_p,t^p))]
$$$$
:=[(a\oplus b,u\oplus v,\oplus_{p=1}^{\infty}(\underline{s_p+t_p},\overline{s^p+t^p})],
$$
where $$\underline{s_p+t_p}=(\theta_{I_aI_{a\oplus b}}\otimes {\rm id}_{C_0(W_p)})\circ s_p\oplus (\theta_{I_bI_{a\oplus b}}\otimes {\rm id}_{C_0(W_p)})\circ t_p$$ and
$$\overline{s^p+t^p}=(S\theta_{I_aI_{a\oplus b}}\otimes {\rm id}_{C_0(W_p)})\circ s^p \oplus (S\theta_{I_bI_{a\oplus b}}\otimes {\rm id}_{C_0(W_p)})\circ t^p.$$

Note that $[(0,1_{\mathbb{C}},\oplus_{p=1}^{\infty}(0,0))]$ is the neutral element and we obtain a partially ordered monoid $(\mathcal{T}(A),+,\leq)$.
\end{definition}
\begin{theorem}
Let $A$ be a  separable $C^*$-algebra of stable rank one. We have a ${\rm \underline{Cu}}$-isomorphism
$$\xi:\,\mathcal{T}(A)\to {\rm \underline{Cu}}(A).$$
\end{theorem}
\begin{proof}
It has been shown that ${\rm \underline{Cu}}(A)$ is an object in  ${\rm \underline{Cu}}$, and hence, ${\rm \underline{Cu}}(A)$ is also an object in ${\rm {Cu}}^\sim$.

Set
$$
\xi([(a,u,\oplus_{p=1}^{\infty}(s_p,s^p))])=([a],\mathfrak{u},\oplus_{p=1}^{\infty}
(\mathfrak{s}_p,\mathfrak{s}^p)),
$$
where
$a\in (A\otimes\mathcal{K})_+$, $u\in \mathcal{U}(I_a^\sim)$, $s_p\in {\rm Hom}(q\mathbb{C}, I_a\otimes C_0(W_p))$, $s^p\in {\rm Hom}(q\mathbb{C}, SI_a\otimes C_0(W_p))$,
$[a]\in {\rm Cu}(A),$ $\mathfrak{u}=[u]\in {\rm K}_1(I_a),$ $\mathfrak{s}_p=[s_p]\in {\rm K}_0(I_a;\mathbb{Z}_p)$ and $\mathfrak{s}^p=[s_p]\in {\rm K}_1(I_a;\mathbb{Z}_p).$

It is immediate that $\xi$ is a well-defined set map from $\mathcal{T}(A)$ to ${\rm \underline{Cu}}(A)$, which is order preserving and injective.
The fact that  ${\rm \underline{K}}(\theta_{I_aI_b})=\delta_{I_aI_b}$ for $a\lesssim_{Cu} b$ implies that $\xi$ is additive.

We show the surjectivity.
Given any $(x,\mathfrak{u},\oplus_{p=1}^{\infty}(\mathfrak{s}_p,\mathfrak{s}^p))\in {\rm \underline{Cu}}(A)$, by Definition \ref{cu total def},
we have
$$
  (x,\mathfrak{u},\oplus_{p=1}^{\infty}(\mathfrak{s}_p,\mathfrak{s}^p))\in{\rm Cu}_f(I_x)\times {\rm K}_1(I_x)\times\bigoplus_{p=1}^{\infty} {\rm K}_* (I_x; \mathbb{Z}_p).
  $$
Then we can lift $x$ to $a\in (A\otimes \mathcal{K})_+$ with
$x=[a]$ and $I_x=I_a$.
Next we lift $\mathfrak{u}$ to $u\in \mathcal{U}(I_a^\sim)$ with $[u]_{{\rm K}_1(I_a)}=\mathfrak{u}$.
For any $p\geq 1$, we have
$$
{\rm K}_* (I_a; \mathbb{Z}_p)\cong {\rm K}_0(I_a;\mathbb{Z}_p)\oplus {\rm K}_1(I_a;\mathbb{Z}_p)$$
and
$$
\mathfrak{s}_p\in {\rm K}_0(I_a;\mathbb{Z}_p),\quad\mathfrak{s}^p\in {\rm K}_1(I_a;\mathbb{Z}_p).$$

Then by \ref{cutnz new}, there exist $s_p\in {\rm Hom}(q\mathbb{C}, I_a\otimes C_0(W_p))$, $s^p\in {\rm Hom}(q\mathbb{C}, SI_a\otimes C_0(W_p))$, satisfying that
$$
[s_p]=\mathfrak{s}_p\in {\rm K}_0(I_b;\mathbb{Z}_p),\quad[s^p]=\mathfrak{s}^p\in {\rm K}_1(I_b;\mathbb{Z}_p).
$$
Now we have
$$\xi([(a,u,\oplus_{p=1}^{\infty}(s_p,s^p))])=
(x,\mathfrak{u},\oplus_{p=1}^{\infty}(\mathfrak{s}_p,\mathfrak{s}^p)).$$

That is, we have $\xi$ is an ordered monoid isomorphism, by \cite[Lemma 4.1]{L1}, $\xi$ becomes a ${\rm Cu}^\sim$-isomorphism and $\mathcal{T}(A)\in {\rm Cu}^\sim$. This implies that
$$
{\rm Gr}(\xi_c):{\rm Gr}(\mathcal{T}(A)_c)\rightarrow {\rm Gr}({\rm \underline{ Cu}}(A)_c)
$$
is a group isomorphic map and in fact, ${\rm Gr}(\xi_c)$ is the identity map of ${\rm \underline{K}}(A)$, which is of course $\Lambda$-linear. This completes the proof.

\end{proof}
\begin{remark}
Thus, using this new picture, the  ${\rm \underline{Cu}}$-morphism ${\rm \underline{Cu}}(\psi)$ induced by a homomorphism $\psi$ in Remark \ref{mor remark} and Theorem \ref{Cu total thm} can also be written as
$$
{\rm \underline{Cu}}(\psi)([(a,u,\oplus_{p=1}^{\infty}(s_p,s^p))])$$
$$=[(\psi(a),\psi^\sim(u),\oplus_{p=1}^{\infty}
( (\psi_{I_aI_{\psi(a)}}\otimes{\rm id}_{ C_0(W_p)})\circ s_p,(S\psi_{I_aI_{\psi(a)}}\otimes{\rm id}_{ C_0(W_p)})\circ s^p))].
$$
 \end{remark}

\section{Continuity of the functor \underline{Cu}}


It is proved in section 3.8 of \cite{L1} that ${\rm Cu}_1$ is a  continuous functor, now we show that $\underline{\rm Cu}$ is also a continuous functor from the category of $C^*$-algebras of stable rank one to the category ${\rm \underline{Cu}}$.

\begin{definition}\label{eve-ins}\rm
Let $S_1\xrightarrow{\beta_{12}} S_2\xrightarrow{\beta_{23}} \cdots$ be a sequence in the category ${\rm Cu}^\sim$. We say $(s_1,s_2,\cdots)$ is an eventually-increasing sequence if, $s_1\in S_1$, $s_2\in S_2,\cdots$ and there exists $k\in \mathbb{N}$ such that for any $j>i>k$, 
$\beta_{ij}(s_i)\leq s_j$. Denote $ \mathcal{S}$ the collection of all the eventually-increasing sequences.

Let $(s_i),(t_i)\in \mathcal{S}$, define the addition operation
$$
(s_i)+(t_i)=(s_i+t_i),
$$
(it is easy to see that $(s_i+t_i)$ belongs to $\mathcal{S}$) and the pre-order relation
 $$
 (s_i)\leq (t_i),
 $$
if there exists $k\in \mathbb{N}$ such that for any $i\geq k$ and $x\in S_i$ with $x\ll s_i$, there is an $m\geq i$ such that $\beta_{ij}(x)\ll t_j$ (in $S_j$) for any $j\geq m$. We say $ (s_i)\sim (t_i)$, if $ (s_i)\leq (t_i)$ and $ (t_i)\leq (s_i)$.

For convenience, for any $i\in\mathbb{N}$ and $s\in S_i$, we will denote all those $\beta_{ij}(s)$ by just $s$ for all $j\geq i$.
\end{definition}
\begin{remark}
  Every increasing sequence defined in \ref{Ell inc} is an eventually-increasing sequence. Every eventually-increasing sequence is equivalent to a sequence of the form
  $$
  (0,0,\cdots,0,s_k,s_{k+1},\cdots),
  $$
  where $s_k\leq s_{k+1}\leq \cdots$. We may denote this sequence by $(s_i)_{i\geq k}$ (the notation $i\geq k$ means  this sequence is increasing from $i=k$).

The eventually-increasing for ${\rm Cu}^\sim$-category is inspired by \cite[Theorem 2]{CEI}.
Note that for particulary cases in ${\rm Cu}^\sim$-category, we may even have $\phi_{k_0-1\,k_0}$ $(S_{k_0-1})\cap s_{k_0}^\leq=\varnothing,$ where $s_{k_0}^\leq\triangleq\{s\in S_{k_0}:\,s\leq s_{k_0}\}$. This means the notion of eventually-increasing is necessary.

To build such an example, one can pick $S_1=\mathbb{N}\cup \{\infty\}$ and
$$S_i=(((\mathbb{N}\backslash\{0\})\cup \{\infty\})\times \mathbb{Z}^{i-1})\cup\{\underbrace{(0,0,\cdots,0)}_i\},\quad {\rm for~ } i\geq 2.$$
For $(m_1,m_2,\cdots,m_i),~(n_1,n_2,\cdots,n_i)\in S_i$,
we say $$(m_1,m_2,\cdots,m_i)\leq(n_1,n_2,\cdots,n_i),$$ if $m_1\leq n_1$ and $m_k=n_k$ for any $k=2,3,\cdots,i$.
Set $$\phi_{ij}((m_1,m_2,\cdots,m_i))=(m_1,m_2,\cdots,m_i,0,\cdots,0)\quad {\rm for }\quad  j> i\geq 1.$$
Then $(S_i,\phi_{ij})$ forms a sequence in the category ${\rm Cu}^\sim$.

Let $s_j=(1,0,\cdots,0,1)$. Now we have $\phi_{ij}(S_{i})\cap s_{j}^\leq=\varnothing$, $j>i\geq1$.
This  is  the reason we use the ``eventually-increasing sequence'' to replace the ``increasing sequence'' for the inductive system in ${\rm Cu}^\sim$-category. (Denote $X_i$ to be the one point union of $i-1$ circles, one can pick
$S_i'={\rm {Cu}_1}(C(X_i))$, for $i\geq 1$, which is also a similar example.)
\end{remark}
\begin{proposition}\label{compatibility leq}
With the assumptions in Definition \ref{eve-ins}, the pre-order relation on $\mathcal{S}$ is reflexive, transitive and compatible with the addition.
\end{proposition}
\begin{proof}
The reflexivity and transitivity  of the pre-order relation on $\mathcal{S}$ is immediate, we only need to check the compatiblity with the addition.

Suppose we have $(s_i)\leq (g_i)$ and $(t_i)\leq (h_i)$ in $\mathcal{S}$ and we will show that $(s_i+t_i)\leq (g_i+h_i)$.
Assume that $(s_i),~(t_i)$ and $(s_i+t_i)$ are all increasing from $i=k$. Then given any $i\geq k$ and $x\ll s_i+t_i$.
Note that $s_i,~t_i\in S_i$, there exist two rapidly increasing sequences $(s_i^n)_n$ and $(t_i^n)_n$ in $S_i$ with
$$
\sup_{n \in \mathbb{N}} s_i^n=s_i\quad {\rm and}\quad \sup_{n \in \mathbb{N}} t_i^n=t_i.
$$
Then denote $l_i^n=s_i^n+t_i^n$, we have
$$
\sup_{n \in \mathbb{N}}l_i^n=s_i+t_i.
$$
There exists an $n_0$ such that for all $n\geq n_0$, we have $x\leq l_i^n$.
In particular,
$$
x\leq s_i^{n_0}+t_i^{n_0}\ll s_i^{n_0+1}+t_i^{n_0+1}.
$$
Now $s_i^{n_0+1}\ll s_i$ and $t_i^{n_0+1}\ll t_i$ in $S_i$. By Definition \ref{eve-ins}, we may find a common $m\geq i$
such that $s_i^{n_0+1}\ll g_m$ and $t_i^{n_0+1}\ll h_m$ in $S_m$. Then
$$
x\ll s_i^{n_0+1}+t_i^{n_0+1}\ll g_m+h_m.
$$
Hence,
$$
(s_i+t_i)\leq (g_i+h_i).
$$
This ends the proof.

\end{proof}
With a similar proof of the first part of \cite[Theorem 2]{CEI}, we have the following:

\begin{lemma}\label{fang ell}
Let $S_1\rightarrow S_2\rightarrow\cdots $ be a sequence in the category ${\rm Cu}^\sim$. Then
$$
{\rm Cu^\sim-}\lim_{\longrightarrow}S_i\cong \mathcal{S}/\sim.
$$
\end{lemma}
\begin{proof}
We divide the proof into several steps.

{\bf Step 0 :} As a trick, let us show that any eventually-increasing sequence is equivalent with an eventually-$\ll$-increasing (eventually and rapidly increasing) sequence.
That is, given an eventually-increasing sequence $(s_i)$, there is an eventually-increasing sequence $(t_i)$ with $t_i\ll t_{i+1}$ for all large enough $i$ satisfying that
$$
 (s_i)\sim (t_i).
 $$
To achieve this, we can assume that the given eventually-increasing sequence $(s_i)$ ($s_i$ is increasing from $i=k$) is an increasing sequence ($k=1$), otherwise, we just begin constructing $(t_i)$ from the $k$th coordinate.

Note that for each $i\geq1$,
$s_i\in S_i$. As $S_i$ belongs to the $ {\rm Cu^\sim}$-category, by (O2), there exists a $\ll$-increasing sequence $\left(s_{i}^n\right)_{n \in \mathbb{N}}$ in $S_i$ such that $\sup_{n \in \mathbb{N}} s_{i}^n=s_i$.
For $i=1$, we keep the $\ll$-increasing sequence $\left(s_{1}^n\right)$ in $S_1$ fixed;
for $i=2$, from the fact $s_{1}^{n}\ll s_1\leq s_2$, we can pick a sub-sequence $\left(s_{2}^{n_j}\right)$ of $\left(s_{2}^n\right)$ in $S_2$ such that
$$
s_{1}^{j}\ll s_{2}^{n_j},\quad \forall j\in \mathbb{N},
$$
and we still denote this new sub-sequence  $\left(s_{2}^{n_j}\right)$ by $\left(s_{2}^n\right)$;
inductively,
we can always assume we have
$$
s_{i}^{n}\ll s_{i+1}^{n},\quad \forall i,n\in \mathbb{N}.
$$

Note that
$$
s_{i}^{i}\ll s_{i+1}^{i}\ll s_{i+1}^{i+1},\quad \forall i\in \mathbb{N},
$$
By choosing the diagonal sequence,
$$
t_i=s_i^i,\quad \forall i\in \mathbb{N}
$$
we get a rapidly increasing sequence $(t_i)$.

Now we show $(s_i)\sim (t_i)$. On one hand, if we have $s\ll t_i$ for some $i$, we get $s\ll s_i^i\ll s_i\leq s_j$, for all $j\geq i$. That is, $(t_i)\leq (s_i)$.
On the other hand, if we have $s\ll s_i$ for some $i$, from the above construction ($\sup_{n \in \mathbb{N}} s_{i}^n=s_i$), there is a large $n_0$ ($n_0\geq i$), such that  $s\ll s_i^{n_0}$. Then $s\ll s_i^{n_0}\ll s_{n_0}^{n_0}\ll s_{j}^{j}=t_j$, for all $j\geq n_0$. Hence, $(s_i)\leq (t_i)$.

Let us now show that $\mathcal{S}/\sim$ is an object in ${\rm Cu}^\sim$-category. This means we must show that $\mathcal{S}/\sim$ is an ordered monoid (Step 1), that each increasing sequence in $\mathcal{S}/\sim$ has a supremum (Step 2---(O1)), that each element in $\mathcal{S}/\sim$ is a supremum of a rapidly increasing sequence (Step 3---(O2)), and finally, the relations $\leq$ and $\ll$ and the operation of passing to the supremum of an increasing sequence are compatible with addition (Step 4---not only (O3), (O4)).

{\bf Step 1 :} We have a reflexive and transitive relation $\leq$ on $\mathcal{S}$, then the induced relation on $\mathcal{S}/\sim$, which we still write as $\leq$ is also reflexive and transitive, i.e., $\leq$ is an order relation on $\mathcal{S}/\sim$. From the fact that the pre-order relation $\leq$ and addition $+$ on $\mathcal{S}$ are compatible (Proposition \ref{compatibility leq}), we have the induced addition $+$ on $\mathcal{S}/\sim$ is well-defined.
Of course, the class of $(0,0,0,\cdots)$ is the unique zero element and this proves that $\mathcal{S}/\sim$ is an ordered monoid (not necessarily positive).

{\bf Step 2 :} Suppose that $s^1\leq s^2\leq s^3\leq \cdots$ is an increasing sequence in $\mathcal{S}/\sim$. Then by Step 0, for each $s^i$,
we pick an eventually-$\ll$-increasing sequence $(s_n^i)_{n\geq k_i}$ such that $[(s_n^i)]=s^i$, where $s_n^i\in S_n$ for any $n$. (Here, we assume that each $(s_n^i)_{n\geq k_i}$ is $\ll$-increasing from $n=k_i$.) For $(s^1_n)$,  it is $\ll$-increasing for $i\geq i_1=k_1$, then $s_{i_1}^1\ll s_{i_1+1}^1$. As $s^1\leq s^2$, there exists an integer $i_2\geq \max\{i_1+1,k_2\}$, such that
$$
s_{i_1}^1\ll s_{i_2}^2\,\,\,\,({\rm in}\,\,S_{i_2}).
$$
Next, we begin with the elements $s_{i_2}^1$ and $s_{i_2}^2$. 
Since $s^1\leq s^3$ and $s^2\leq s^3$, there exist an $i_3\geq \max\{ i_2+1,k_3 \}$ such that
$$
s_{i_2}^1\ll s_{i_3}^3\quad {\rm and}\quad s_{i_2}^2\ll s_{i_3}^3 \,\,\,\,({\rm in}\,\,S_{i_3}).
$$
Inductively, for any $m\in \mathbb{N}$,
we have $i_{m+1}\geq \max\{ i_m+1,k_{m+1}\}$ with
$$
s_{i_m}^j\ll s_{i_{m+1}}^{m+1},\quad j=1,2,\cdots,m.
$$

 Set
$$
s_j=
\begin{cases}
  0, & \mbox{if } 1\leq j< i_1 \\
  s_{i_1}^1\,\,({\rm in}\,\,S_j), & \mbox{if } i_1\leq j< i_2 \\
  s_{i_2}^2\,\,({\rm in}\,\,S_j), & \mbox{if } i_2\leq j< i_3 \\
  \vdots & \vdots
\end{cases}.
$$
Then $(s_1,s_2,\cdots)$ is an eventually-increasing sequence (not necessary a rapid one), we denote $s$ to be the class of this sequence in $\mathcal{S}/\sim$ and claim that $s=\sup_{i\in \mathbb{N}} s^i$.

 Note first that, for each $j$ and each $n$ ($n\geq k_j$), $s_n^j\leq s_{i_{j+n}}$, hence, $s^i\leq s$, and second, that if $ s^1,s^2,\cdots\leq t=(t_1,t_2,\cdots)$, then for each $i$, if $r\ll s_i$ in $S_i$, eventually $r\ll t_j$, so $s\leq t$.


{\bf Step 3 :} We say $t\ll s$ in  $\mathcal{S}/\sim$, if for any increasing sequence $(s^n)$ in $\mathcal{S}/\sim$ with $\sup_{n\in \mathbb{N}}s^n=s$, there exists an $n_0\in\mathbb{N}$, such that $t\leq s^n$ for all $n\geq n_0$.

Suppose we have an $s\in\mathcal{S}/\sim$, then by Step 0, we have an eventually-$\ll$-increasing sequence $(s_i)$ representing $s$, i.e., $[(s_i)]=s$. We assume that $(s_i)$ is  $\ll$-increasing from $i=k$ and let

$$
s^1=[(\underbrace{0,0,\cdots,0}_{k},s_{k+1},s_{k+1},\cdots)],
$$
$$
s^2=[(\underbrace{0,0,\cdots,0}_{k},s_{k+1},s_{k+2},s_{k+2},\cdots)],
$$
$$
\cdots
$$
$$
s^n=[(\underbrace{0,0,\cdots,0}_{k},s_{k+1},s_{k+2},\cdots,s_{k+n},s_{k+n},\cdots)],
$$
$$
\cdots
$$
Now we prove that  the sequence $(s^1,s^2,\cdots)$ is a rapidly increasing sequence in $\mathcal{S}/\sim$ and has supremum $s$.

To see that it is rapidly increasing. Suppose that we have an increasing sequence $(t^j)$ in  $\mathcal{S}/\sim$ with $\sup_{j\in \mathbb{N}} t^j= s^{n+1}$. For each $t^j\in \mathcal{S}/\sim$,
we pick an eventually-$\ll$-increasing sequence $(t_1^j,t_2^j,t_3^j\cdots)$ from this class, where $t_i^j\in S_i$.
Then set $i_1=k+1$, we repeat the construction in Step 2, we have
$$
[(\underbrace{0,0,\cdots,0}_{i_1-1},\underbrace{t^1_{i_1},t^1_{i_1},
\cdots,t^1_{i_1}}_{i_2-i_1},\cdots,\underbrace{t^j_{i_j},t^j_{i_j},
\cdots,t^j_{i_j}}_{i_{j+1}-i_j},\cdots)]= s^{n+1}.
$$
In particular, $$ (\underbrace{0,0,\cdots,0}_{k},s_{k+1},\cdots,s_{k+n+1},s_{k+n+1},\cdots)
\leq
(\underbrace{0,0,\cdots,0}_{i_1-1},\underbrace{t^1_{i_1},t^1_{i_1},\cdots,t^1_{i_1}}_{i_2-i_1},\cdots).
$$

Since $s_{k+n}\ll s_{k+n+1}$, by Definition \ref{eve-ins}, we have an $i_m$ such that
$$
s_{k+n}\ll t^m_{i_m} \quad ({\rm in}\,\, S_{i_m}).
$$
Then for all $j\geq m$, we have
$$
s_{k+n}\ll  t^m_{i_m}\ll t^j_{i_j},
$$
which implies $s^n\leq t^j$ for all $j\geq m$. Hence $$s^n\ll s^{n+1}.$$



Now we show that $\sup_{n\in \mathbb{N}} s^n= s$. Note that $s_i\ll s_{i+1}$ for all $i\geq k$,
by the construction in Step 2, $\sup_{n\in \mathbb{N}} s^n$ contains a representing element as
$$
(\underbrace{0,0,\cdots,0}_{k},s_{k+1},s_{k+2},s_{k+3},\cdots),
$$
which is equivalent with $(s_i)$.  We do have $\sup_{n\in \mathbb{N}} s^n= s$.

Here, we add one more remark. Suppose that $(r_1,r_2,r_3,\cdots)$ is an eventually-$\ll$-increasing sequence in $\mathcal{S}$ with $r_i\in S_i$,
$[(r_1,r_2,r_3,\cdots)]=r$ and $(r_i)$ is $\ll$-increasing from $i=k$. If we regard $r_i$ as the element $[(\underbrace{0,\cdots,0}_{i-1},r_i,r_i,\cdots)]$ in $\mathcal{S}/\sim$, then
$$
(r_k,r_{k+1},r_{k+2},\cdots)
$$
is a $\ll$-increasing sequence in $\mathcal{S}/\sim$, and hence by the above construction, it has supremum
$$
[(\underbrace{0,0,\cdots,0}_{k-1},r_k,r_{k+1},\cdots)]=r,
$$
and we may simply write $r=\sup_{i\geq k}r_i$.

{\bf Step 4 :}
Firstly, as we mentioned in Step 1 that the addition on $\mathcal{S}/\sim$ is well-defined, the compatibility of the relation $\leq$ and addition on $\mathcal{S}/\sim$ just comes from the compatibility of the relation $\leq$ and addition on $\mathcal{S}$ (Proposition \ref{compatibility leq}).

Secondly, we show that $\mathcal{S}/\sim$ satisfies (O4).
Suppose that we have two increasing sequence $(s^i),~(t^i)$ in  $\mathcal{S}/\sim$ with $\sup_{i\in \mathbb{N}} s^i= s$ and $\sup_{i\in \mathbb{N}} t^i= t$. Note that $(s^i+t^i)$ is also an increasing sequence in  $\mathcal{S}/\sim$.
We do the construction in  Step 2 to $(s^i)$, $(t^i)$ and $(s^i+t^i)$, we may find a common sequence $(i_n)_{n\in \mathbb{N}}$ such that
$$
s=[(\underbrace{0,0,\cdots,0}_{i_1-1},\underbrace{s^1_{i_1},\cdots,s^1_{i_1}}_{i_2-i_1},
\underbrace{s^2_{i_2},\cdots,s^2_{i_2}}_{i_3-i_2},\cdots,\underbrace{s^m_{i_m},\cdots,s^m_{i_m}}_{i_{m+1}-i_m},\cdots)],
$$

$$
t=[(\underbrace{0,0,\cdots,0}_{i_1-1},\underbrace{t^1_{i_1},\cdots,t^1_{i_1}}_{i_2-i_1},
\underbrace{t^2_{i_2},\cdots,t^2_{i_2}}_{i_3-i_2},\cdots,\underbrace{t^m_{i_m},\cdots,t^m_{i_m}}_{i_{m+1}-i_m},\cdots)]
$$
and
$$
\sup_{i\in \mathbb{N}}\{s^i+t^i\}=[(\underbrace{0,\cdots,0}_{i_1-1},\underbrace{s^1_{i_1}+t^1_{i_1},\cdots}_{i_2-i_1},
\cdots,\underbrace{s^m_{i_m}+t^m_{i_m},\cdots,s^m_{i_m}+t^m_{i_m}}_{i_{m+1}-i_m},\cdots)].
$$
At once, from the definition of addition on $\mathcal{S}/\sim$, we get $$\sup_{i\in \mathbb{N}}\{s^i+t^i\}=s+t=\sup_{i\in \mathbb{N}}s^i+\sup_{i\in \mathbb{N}}t^i.$$

Thirdly, we prove the compatibility of $\ll$ with addition.
Assume $s\ll g$ and $t\ll h$ in $\mathcal{S}/\sim$, we will show that $s+t\ll g+h$.
We do the construction in Step 3 for both $g$ and $h$, we have two rapidly increasing sequences
$$
g^n=[(\underbrace{0,0,\cdots,0}_{k},g_{k+1},g_{k+2},\cdots,g_{k+n},g_{k+n},\cdots)],\quad n\in \mathbb{N}
$$
and
$$
h^n=[(\underbrace{0,0,\cdots,0}_{k},h_{k+1},h_{k+2},\cdots,h_{k+n},h_{k+n},\cdots)],\quad n\in \mathbb{N}.
$$
with $\sup_{n\in \mathbb{N}}g^n=g$ and $\sup_{n\in \mathbb{N}}h^n=h$ in $\mathcal{S}/\sim$.
Then there is an $n_0$ such that
$$s\leq g^{n_0}\ll g^{n_0+1}\leq g$$
and
$$t\leq h^{n_0}\ll h^{n_0+1}\leq h.$$
Note that
$$g^{n_0}+h^{n_0}=[(\underbrace{0,0,\cdots,0}_{k},f_{k+1},f_{k+2},\cdots,f_{k+n_0},f_{k+n_0},\cdots)]$$
and
$$g^{n_0+1}+h^{n_0+1}=[(\underbrace{0,0,\cdots,0}_{k},f_{k+1},f_{k+2},
\cdots,f_{k+n_0+1},f_{k+n_0+1},\cdots)],$$
where $f_n=g_n+h_n$ for any $n\in\mathbb{N}$.
Still from proof of Step 3, we have $g^{n_0}+h^{n_0}\ll g^{n_0+1}+h^{n_0+1}$.
Then from compatibility of the order relation and addition
 on $\mathcal{S}/\sim$ (the first part of Step 4), we have
$$
s+t\leq g^{n_0}+h^{n_0}\ll g^{n_0+1}+h^{n_0+1}\leq g+h.
$$


Now let us complete the proof by proving that the object $\mathcal{S}/\sim$ in ${\rm Cu}^\sim$-category is the inductive limit of the given sequence $S_1\rightarrow S_2\rightarrow\cdots $ in this category. We must show that for every object $T$ in ${\rm Cu}^\sim$-category and every compatible sequence of maps $S_1\to T$, $S_2\to T$, $\cdots$, there exists a unique compatible map $\mathcal{S}/\sim\to T$.
$$
\xymatrix{
&&&T\\
S_1\ar[r]\ar[urrr] & S_2\ar[r]\ar[urr]&\cdots
\ar[r]& \mathcal{S}/\sim \ar[u]_{!}
}
$$
Of course, for this to make sense we must have maps $S_i\to \mathcal{S}/\sim$ for all $i$ compatible with maps $S_i\to S_{i+1}$. For each $i$, and each $s\in S_i$,
note that the sequence
$$
(\underbrace{0,0,\cdots,0}_{i-1},{s,s,\cdots,s},\cdots)
$$
is an eventually-increasing sequence and therefore represents an element of $\mathcal{S}/\sim$. For each fixed $i$, note that for any $j\geq i$, the eventually-increasing sequence $(\underbrace{0,0,\cdots,0}_{j-1},s,s,\cdots,s,\cdots)$ is equivalent to the one with $j=i$, this follows immediately from Definition \ref{eve-ins}.
This shows that the maps  $S_1\to \mathcal{S}/\sim$, $S_2\to \mathcal{S}/\sim$, $\cdots$ obtained in this way are compatible with the given sequence. It is easy to see that all these maps are ${\rm Cu}^\sim$-morphisms.

Note that the maps $S_i\to T$ preserve the order relation and are compatible as set maps. The definition of a (set) map $\mathcal{S}/\sim\to T$ is immediate if one restricts to eventually-$\ll$-increasing representative sequences
for elements of $\mathcal{S}/\sim$ (Step 0).
(If $(s_1,s_2,\cdots)$ represents $s$ with $s_i\in S_i$ and $s_k\ll s_{k+1}\ll s_{k+2}\ll\cdots$ for some $k$, let us aim to map $s$ into
the supremum of increasing sequence in $T$ consisting of the images of $s_k, s_{k+1}, s_{k+2},\cdots$ by the maps $S_{k}\to T$, $S_{k+1}\to T$, $S_{k+2}\to T$, $\cdots$; this of course makes sense if the sequence $(s_1,s_2,\cdots)$ is just  eventually-increasing.
If $(s_1',s_2',\cdots)$ is a second eventually-$\ll$-increasing representative sequence of $s$, assume that $s_k'\ll s_{k+1}'\ll s_{k+2}'\ll\cdots$ for the same $k$ as above.
Then for each $i\geq k$, as $s_i\ll s_{i+1}$, there is a large enough $m$ ($m\geq k$) we have $s_i\ll s_m'$. This implies for any $i\geq k$, we have $s_i\leq \sup_{j\geq k} s_j'$ (all these are images and elements in $T$), and hence
$\sup_{i\geq k} s_i\leq \sup_{j\geq k} s_j'$ in $T$.
By symmetry, the suprema of the images of
$s_k, s_{k+1}, s_{k+2},\cdots$ and of $s_k', s_{k+1}', s_{k+2}',\cdots$ in $T$ are equal.

Let us check that the set map $\mathcal{S}/\sim\to T$ thus defined is compatible with the maps $S_i\to T$ (Step 5, i.e., that the diagram is commutative as a diagram of set maps), and that it is a morphism in the category ${\rm Cu}^\sim$ (Step 6), and that it is unique with these properties (Step 7).

{\bf Step 5 :}
To show compatibility of $S_i\to T$ with $\mathcal{S}/\sim\to T$, we must show, for each fixed $i$,
that if $s\in S_i$ then the image of $s$ in $T$ by the given map $S_i\to T$ is the same as the image of $s$ in $T$ by the composed map
$s_i\to \mathcal{S}/\sim\to T$. By definition, the image of $s$ in $\mathcal{S}/\sim$ is represented by the sequence $(\underbrace{0,\cdots,0}_{i-1},s,s,\cdots)$;
However, in order to compute the image of this element of $\mathcal{S}/\sim$ in $T$, we must represent it by an eventually-$\ll$-increasing sequence:
let us use the sequence $(\underbrace{0,\cdots,0}_{i-1},r_i,r_{i+1},\cdots)$, where $r_i\ll r_{i+1}\ll r_{i+2}\ll\cdots$ is a rapidly increasing sequence
in $S_i$ with supremum $s$.
By definition, the corresponding element in $T$ is the supremum of the (increasing sequence of) images
of the elements $r_i\in S_i,~r_{i+1}\in S_{i+1},\cdots$ by the maps $S_i\to T$, $S_{i+1}\to T,\cdots$, equivalently (by commutativity) of the images of $r_i,r_{i+1}\cdots\in S_i$ by the map $S_i\to T$, and as this map preserves increasing sequential suprema, the supremum in
$T$ in question is just the image of $s$, by the map $S_i\to T$, as desired.

{\bf Step 6 :}
To show that the map $\mathcal{S}/\sim\to T$ is a ${\rm Cu}^\sim$-morphism, we must show that it preserves addition (1), preserves the order relation (2), preserves suprema of increasing sequences (3), and preserves the order-theoretic relation $\ll$ listed in Step 3 (4),
in terms of the two notions just mentioned (the order relation $\leq$ and the operation of sequential increasing supremum).
Let us address these issues, briefly, in turn.

 (1) Given two eventually-$\ll$-increasing sequences $(r_i)$ and $(s_i)$ with $r_i, s_i\in S_i$ for each $i$, we may suppose both of these two sequences are  $\ll$-increasing from $i=k$. To check that the sum in $\mathcal{S}/\sim$ maps into the sum of images in $T$ it is enough to recall what these images are, and that the operation of passing to the supremum of an increasing sequence in $T$ is compatible with addition in $T$.

(2) To check that the relation $[(r_1,r_2,\cdots)]\leq[(s_1,s_2,\cdots)]$ in $\mathcal{S}/\sim$
(with $(r_i)$, $(s_i)$ are eventually-$\ll$-increasing) leads to the same relation between the images in $T$, we may suppose both of these two sequences are  $\ll$-increasing from $i=k$. Then for any $i\geq k$, as $r_i\ll r_{i+1}$, by Definition \ref{eve-ins}, we have a large $m$ ($m\geq k$) with
$r_i\ll s_m$ in $S_m$, from the properties of the map $S_m\rightarrow T$, the image of $r_i$ in $T$ is compactly contained in the image of $s_m$ in $T$, 
which means that $r_i\leq \sup_{j\geq k}s_j$ in $T$. Further, we have $\sup_{i\geq k}r_i\leq \sup_{j\geq k}s_j$ in $T$, in other words, $[(r_1,r_2,\cdots)]\leq[(s_1,s_2,\cdots)]$ in $T$. 

(3) Using the same notations in Step 2,  let $s^1\leq s^2\leq s^3\leq \cdots$ be an increasing sequence in $\mathcal{S}/\sim$ with supremum $s$, and for each $i$, let $(s_n^i)_{n\geq k_i}$ be an eventually-$\ll$-increasing sequence such that $[(s_n^i)]=s^i$, where $s_n^i\in S_n$ for any $n$. We have an eventually-increasing sequence $(s_1,s_2,\cdots)$ (not necessary a rapid one) as follows: 
$$
(s_n)=(\underbrace{0,0,\cdots,0}_{i_1-1},\underbrace{s^1_{i_1},\cdots,s^1_{i_1}}_{i_2-i_1},
\underbrace{s^2_{i_2},\cdots,s^2_{i_2}}_{i_3-i_2},\cdots,\underbrace{s^m_{i_m},\cdots,s^m_{i_m}}_{i_{m+1}-i_m},\cdots),
$$
where $s_n\in S_n$ and $[(s_1,s_2,\cdots)]=s$.
We mention that $s^1_{i_1}\ll s^2_{i_2}\ll s^3_{i_3}\ll\cdots$. Suppose we have an eventually-$\ll$-increasing sequence $(t_n)$ ($\ll$-increasing from $n=k$) with  $[(t_n)]=s$ (Step 0), denote $\dot{s}$ by the image of $s$ in $T$, use the compatibility of the maps $S_i\rightarrow T$ and Definition \ref{eve-ins}, we can prove that $\sup_{j\geq 1}\dot{s}_{i_j}^j= \sup_{n\geq k}\dot{t}_n=\dot{s}$ in $T$ and that $\sup_{j\geq 1}\dot{s}_{i_j}^j=\sup_{j\geq 1}\sup_{n\geq i_j}\dot{s}^j_n=\sup_{j\geq 1}\dot{s}^j$ in $T$, i.e., the supremum of sequence of the images of $s^j$ in $T$ is indeed the image of $s$ in $T$.

(4) 
Let $(r_1,r_2,\cdots)$ and $(s_1,s_2,\cdots)$ be two eventually-$\ll$-increasing sequences with $r_i, s_i\in S_i$ for each $i$, satisfying $r\ll s,$
 where $$
[(r_1,r_2,\cdots)]=r,\quad[(s_1,s_2,\cdots)]=s.$$

We may  also suppose both of these two sequences are  $\ll$-increasing from $i=k$.
We must show that the supremum of the images of $r_k,r_{k+1},\cdots$ in $T$ is compactly contained in the supremum of the images of $s_1,s_2,\cdots$ in $T$, in other words,  if $\dot{r}$ denotes the image of $r=\sup_{i\geq k}r_i$ in $T$, and $\dot{s}$ denotes the image of  $s=\sup_{i\geq k}s_i$ in $T$, then $\dot{r}\ll\dot{s}$.

Since $r\ll s$, we have $r\leq s_i$ in $\mathcal{S}/\sim$ for some $i\geq k$. Since the map $\mathcal{S}/\sim\rightarrow T$ preserves the relation $\leq$ (Step 6(2)), it follows that $\dot{r}\leq \dot{s_i}$ in $T$. Since $s_i\ll s_{i+1}$, not only in $\mathcal{S}/\sim$ but also in $S_{i+1}$, it follows from the given map  $S_{i+1}\rightarrow T$ preserves the relation $\ll$ that $\dot{s_i}\ll \dot{s}_{i+1}$ in $T$. Since $s_{i+1}\leq s$ in $\mathcal{S}/\sim$, we have $\dot{s}_{i+1}\leq \dot{s}$ in $T$, and hence
$\dot{r}\leq \dot{s_i}\ll\dot{s}_{i+1}\leq \dot{s},$
then $\dot{r}\ll\dot{s}$ in $T$.


{\bf Step 7 :}
We have shown that the map $\mathcal{S}/\sim\to T$ preserves $\leq$ relation (Step 6(2)) and  preserves suprema of increasing sequences (Step 6(3)). Note that from Step 0, for any $s\in\mathcal{S}/\sim$, we can pick an eventually-$\ll$-increasing sequences $(s_1,s_2,\cdots)$ ($\ll$-increasing from $i=k$)
to represent $s$. From the last remark in Step 3, we have $\sup_{i\geq k}s_i= s$ in $\mathcal{S}/\sim$. 
Note that the choice of $\dot{ s}_i$ in $T$ is unique by the given map $S_i\rightarrow T$.
From the uniqueness of the $\sup_{i\geq k}\dot{s}_i$ in $T$, we obtain the  uniqueness of the map $\mathcal{S}/\sim\to T$.

\end{proof}

\begin{remark} \label{limitpro}
Let  $(S_i,\beta_{ij})$ be an inductive system in ${\rm Cu}^\sim$-category, and
$
S=\mathcal{S}/\sim.
$
Suppose we have an object $T\in{\rm Cu^\sim-}$category satisfying that for any $j\geq i\in\mathbb{N}$,
the following diagram is commutative.
$$
\xymatrixcolsep{5pc}
\xymatrix{
{\,\, S_i\,\,} \ar[r]^-{\psi_{i}} \ar[rd]_-{\beta_{ij}}
& {\,\, T\,\,}
\\
{\,\,\,\,}
& {\,\,S_j\,\,} \ar[u]^-{\psi_j}}
$$
Then by universal property of inductive limit, there is a unique ${\rm Cu^\sim-}$map $\omega$ such that the following is also commutative for any $i\in\mathbb{N}$.
$$
\xymatrixcolsep{5pc}
\xymatrix{
{\,\, S_i\,\,} \ar[r]^-{\psi_{i}} \ar[rd]_-{\beta_{i\infty}}
& {\,\, T\,\,}
\\
{\,\,\,\,}
& {\,\,S\,\,} \ar[u]^-{\exists !\omega}}
$$

Given $s=(s_i)_{i\geq k_0}\in\mathcal{S}$, for each ${i\geq k_0}$, we have
$$
\beta_{i\infty}(s_i)=[(\beta_{ij}(s_{i}))_{j\geq i}],
$$
and hence from the commutativity,
$$
\psi_i(s_i)=\omega([(\beta_{ij}(s_{i}))_{j\geq i}]).
$$
Then as
$$
[s]=\sup_{i\geq k_0} \beta_{i\infty}(s_i)
$$
and $\omega$ preserves the suprema, we have
$$
\omega([s])=\omega(\sup_{i\geq k_0} \beta_{i\infty}(s_i))=\sup_{i\geq k_0} \omega(\beta_{i\infty}(s_i))=\sup_{i\geq k_0}\psi_i(s_i).
$$
\end{remark}

\begin{proposition}\label{cu total con}
Let $(S_i,\beta_{ij})$ be an inductive system in $\underline{{\rm Cu}}$-category, and
Then $$
{\rm {\rm Cu}^\sim-}\lim_{\longrightarrow}S_i={\rm \underline{Cu}-}\lim_{\longrightarrow}S_i.
$$
\end{proposition}
\begin{proof}
Denote $S={\rm {\rm Cu}^\sim-}\lim\limits_{\longrightarrow}S_i.$
From the last remark of Step 3 in Lemma \ref{fang ell}, we have $$S_c=\lim_{\longrightarrow}(S_i)_c$$ is an algebraic limit, then $${\rm Gr}(S_c)=\lim_{\longrightarrow}{\rm Gr}((S_i)_c)$$ is also an algebraic limit, which becomes a $\Lambda$-module.

Suppose $T$ is an object in $\underline{{\rm Cu}}$-category with compatible $\underline{{\rm Cu}}$-morphisms $\psi_i:\,S_i\to T$. Denote the ${\rm Cu}^\sim$-morphism we obtained in Lemma \ref{fang ell} by $\omega$.
$$
\xymatrix{
&&&T\\
S_1\ar[r]\ar[urrr] & S_2\ar[r]\ar[urr]&\cdots
\ar[r]& S \ar[u]_{\omega}
}
$$
Since each ${\rm Gr}((\psi_i)_c):\,{\rm Gr}((S_i)_c)\to {\rm Gr}(T_c)$ is $\Lambda$-linear and compatible with
each ${\rm Gr}((\beta_{ij})_c):\,{\rm Gr}((S_i)_c)\to {\rm Gr}((S_j)_c)$, we have
${\rm Gr}(\omega_c):\,{\rm Gr}(S_c)\to {\rm Gr}(T_c)$ is $\Lambda$-linear.

\end{proof}

\begin{theorem}
Let $A=\lim\limits_{\longrightarrow}(A_i,\phi_{ij})$ be an inductive system in the category $C^*_{sr1}$. Then
$$
{\rm \underline {Cu}-}\lim_{\longrightarrow}{\rm \underline{Cu}}(A_i)\cong
{\rm \underline{Cu}}(A).
$$
\end{theorem}
\begin{proof}
Since $({\rm \underline{Cu}}(A_i), {\rm \underline{Cu}}(\phi_{ij}))$ is an inductive system in the category ${\rm \underline{Cu}}$,  
denote
$$
S={\rm \underline{Cu}}-\lim_{\longrightarrow}{\rm \underline{Cu}}(A_i).
$$
For ${\rm \underline{Cu}}(A)$ in ${\rm \underline{Cu}}$-category and the compatible sequence of maps
$$
{\rm \underline{Cu}}(\phi_{i\infty}):\,{\rm \underline{Cu}}(A_i)\to{\rm \underline{Cu}}(A),$$
by Proposition \ref{cu total con}, there exists a unique compatible ${\rm \underline{Cu}}$-morphism $\omega:\,S\to {\rm \underline{Cu}}(A)$ such that the following diagram is commutative:
$$
\xymatrixcolsep{5pc}
\xymatrix{
{\,\,{\rm \underline{Cu}}(A_i)\,\,} \ar[r]^-{{\rm \underline{Cu}}(\phi_{i\infty})} \ar[rd]_-{}
& {\,\,{\rm \underline{Cu}}(A)\,\,}
\\
{\,\,\,\,}
& {\,\,S\,\,} \ar[u]^-{\exists !\omega}}
$$
Now we are going to prove that $\omega$ is a ${\rm \underline{Cu}}$-isomorphism.

Firstly, we show that $\omega$ is surjective. For any $ (x,\mathfrak{u},\oplus_{p=1}^{\infty}(\mathfrak{s}_p,\mathfrak{s}^p))\in {\rm \underline{Cu}}(A)$,  by \ref{Ell inc}, there exists an increasing sequence $(x_1,x_2,\cdots)$ with $x_1\in {\rm {Cu}}(A_1)$, $x_2\in {\rm {Cu}}(A_2), \cdots$  such that
$$
x=[(x_1,x_2,\cdots)],
$$
Let $I_{x_i}$ be the ideal of $A_i$  generated by $x_i\in{\rm Cu}(A_i)$,  and let $I_x$ be the ideal of $A$ generated by $x\in{\rm Cu}(A)$. Then
$$
I_{x_1}\to I_{x_2}\to\cdots \to I_x\lhd A,
$$
with $I_x={\rm C}^*-\lim I_{x_i}.$

Note that $(0,\mathfrak{u},\oplus_{p=1}^{\infty}(\mathfrak{s}_p,\mathfrak{s}^p))\in {\underline{\rm K}}(I_x)$, by the continuity of $ \underline{\rm K}$, there exist $k\in \mathbb{N}$ and
$(0,\mathfrak{u}_{k},\oplus_{p=1}^{\infty}(s_{k,p},s^{k,p}))\in {\underline{\rm K}}(I_{x_{k}})$ satisfying
$$
\underline{\rm K}(\phi_{{k\infty}{I_{x_k}I_x}})(0,\mathfrak{u}_{k},\oplus_{p=1}^{\infty}(\mathfrak{s}_{k,p},\mathfrak{s}^{k,p}))=
(0,\mathfrak{u},\oplus_{p=1}^{\infty}(\mathfrak{s}_p,\mathfrak{s}^p))\in {\underline{\rm K}}(I_{x}).
$$
and for any $i>k$, pick
$$
(0,\mathfrak{u}_{i},\oplus_{p=1}^{\infty}(\mathfrak{s}_{i,p},\mathfrak{s}^{i,p})):=\underline{\rm K}(\phi_{{ki}I_{x_k}I_{x_i}})(0,\mathfrak{u}_{k},\oplus_{p=1}^{\infty}(\mathfrak{s}_{k,p},\mathfrak{s}^{k,p}))\in {\underline{\rm K}}(I_{x_i}).
$$
Then we obtain an eventually-increasing sequence
$((x_j,\mathfrak{u}_{j},\oplus_{p=1}^{\infty}(\mathfrak{s}_{j,p},\mathfrak{s}^{j,p})))_j$.
Set
$$
s=((x_j,\mathfrak{u}_{j},\oplus_{p=1}^{\infty}(\mathfrak{s}_{j,p},\mathfrak{s}^{j,p})))_{j\geq k},
$$
By Lemma \ref{fang ell}, we have $[s]\in S$.

Then by the commutativity of
$$
\xymatrixcolsep{5pc}
\xymatrix{
{\,\,{\rm \underline{Cu}}(A_i)\,\,} \ar[r]^-{{\rm \underline{Cu}}(\phi_{i,\infty})} \ar[rd]_-{\Phi_i}
& {\,\,{\rm \underline{Cu}}(A)\,\,}
\\
{\,\,\,\,}
& {\,\,S\,\,} \ar[u]^-{\exists !\omega},}
$$
we get
$$
\omega([s])=
(x,\mathfrak{u},\oplus_{p=1}^{\infty}(\mathfrak{s}_p,\mathfrak{s}^p)).
$$

Secondly, we show $\omega$ is injective.
Let $f=(x_i,\mathfrak{u}_{i},\oplus_{p=1}^{\infty}(\mathfrak{s}_{i,p},\mathfrak{s}^{i,p}))_{i\geq k_1}$ and
$g=(y_j,\mathfrak{v}_{j},\oplus_{p=1}^{\infty}(\mathfrak{t}_{j,p},t^{j,p}))_{j\geq k_2}$
be two eventually-increasing sequences of inductive system $({\rm \underline{Cu}}(A_i),{\rm \underline{Cu}}(\phi_{ij}))$ in the category ${\rm Cu}^\sim$. Suppose $\omega([f])=\omega([g])\in{\rm \underline{Cu}}(A)$,
we need only to prove $f\sim g$.


From Remark \ref{limitpro}, we have
$$
\sup_{i\geq k_1} {\rm \underline{Cu}}(\phi_{i\infty})(x_i,\mathfrak{u}_{i},
\oplus_{p=1}^{\infty}(\mathfrak{s}_{i,p},\mathfrak{s}^{i,p}))=
\sup_{j\geq k_2} {\rm \underline{Cu}}(\phi_{j\infty})(y_j,\mathfrak{v}_{j},
\oplus_{p=1}^{\infty}(\mathfrak{t}_{j,p},t^{j,p})),$$
then
$$
[(x_i)]=\sup {\rm {Cu}}(\phi_{i\infty})(x_i)
=
\sup {\rm {Cu}}(\phi_{j\infty})(y_j)=[(y_j)],\,\,{\rm i.e.},~(x_i)\sim (y_i).
$$
(Here we take $x_i=0$ for $i\leq k_1-1$ and $y_j=0$ for $j\leq k_2-1$.)

Denote $z=[(x_i)]=[(y_i)]\in {\rm Cu}(A)$, then
$$
I_{x_1}\to I_{x_2}\to\cdots \to I_z\lhd A,
$$
$$
I_{y_1}\to I_{y_2}\to\cdots \to I_z\lhd A
$$
and
$$
I_z={\rm C^*-}\lim I_{x_i}={\rm C^*-}\lim I_{y_i}.$$

Now we have
$$
\xymatrixcolsep{3pc}
\xymatrix{
{\,\,\underline{\mathrm{K}}(I_{x_1})\,\,}  \ar[r]^-{\varphi_{12}^*}
& {\,\,\underline{\mathrm{K}}(I_{x_2})\,\,}  \ar[r]^-{}
& {\,\,\cdots\,\,} \ar[r]^-{}
& {\,\,\underline{\mathrm{K}}(I_{z})\,\,}\ar[d]_-{=}
 \\
{\,\,\underline{\mathrm{K}}(I_{y_1})\,\,} \ar[r]_-{\psi_{12}^*}
& {\,\,\underline{\mathrm{K}}(I_{y_2}) \,\,} \ar[r]_-{}
& {\,\,\cdots \,\,} \ar[r]_-{}
& {\,\,\underline{\mathrm{K}}(I_{z})\,\,},}
$$
where the  maps $\varphi_{ij}^*$ and $\psi_{ij}^*$ are in fact $\underline{\mathrm{K}}(\phi_{{ij}{I_{x_i}I_{x_j}}})$ and $\underline{\mathrm{K}}(\phi_{{ij}{I_{y_i}I_{y_j}}})$, respectively.

As  $(x_i)\leq (y_i)$, by \ref{Ell inc}, there exists $k\in \mathbb{N}$ such that for any $i\geq k$ and $c\in {\rm {Cu}}(A_i)$ with $c\ll x_i$, there exists $m_c$ such that ${\rm Cu} (\phi_{ij})(c) \ll y_j$ (in ${\rm {Cu}}(A_j)$), for all $j\geq m_c$.

Suppose we have $(c,\mathfrak{w}_i,\oplus_{p=1}^{\infty}(\mathfrak{r}_{i,p},\mathfrak{r}^{i,p}))\in{\rm \underline{Cu}}(A_i) $ with
$$(c,\mathfrak{w}_i,\oplus_{p=1}^{\infty}(\mathfrak{r}_{i,p},\mathfrak{r}^{i,p}))\ll
(x_i,\mathfrak{u}_{i},\oplus_{p=1}^{\infty}(\mathfrak{s}_{i,p},\mathfrak{s}^{i,p})),$$
then we have
$$
\delta_{I_cI_{x_i}}(\mathfrak{w}_i,\oplus_{p=1}^{\infty}(\mathfrak{r}_{i,p},\mathfrak{r}^{i,p}))
=(\mathfrak{u}_i,\oplus_{p=1}^{\infty}(\mathfrak{s}_{i,p},\mathfrak{s}^{i,p})), \,\,\forall i\geq k.
$$

Now we have the natural commutative diagram:
$$
\xymatrixcolsep{3pc}
\xymatrix{
{\,\,I_{c}\,\,} \ar[d]_-{\phi_{{ij}I_{c} I_{{\rm Cu}(\phi_{ij})(c)}}} \ar[r]^-{}
& {\,\,I_{x_i}\,\,}  \ar[r]^-{}
& {\,\,I_{z}.\,\,}
\\
{\,\, I_{{\rm Cu}(\phi_{ij})(c)}\,\,} \ar[r]^-{}
& {\,\,I_{y_j} \,\,}  \ar[ru]^-{}}
$$
Then the following diagram commutates at the level of $\underline{\rm K}$:
$$
\xymatrixcolsep{3pc}
\xymatrix{
{\,\,\underline{\mathrm{K}}(I_{c})\,\,} \ar[d]_-{\underline{\rm K}(\phi_{{ij}I_{c} I_{{\rm Cu}(\phi_{ij})(c)}})} \ar[r]^-{\delta_{I_c I_{x_i}}}
& {\,\,\underline{\mathrm{K}}(I_{x_i})\,\,}  \ar[r]^-{}
& {\,\,\underline{\mathrm{K}}(I_{z})\,\,}.
\\
{\,\,\underline{\mathrm{K}}(I_{{\rm Cu}(\phi_{ij})(c)}) \,\,} \ar[r]^-{\delta_{I_{{\rm Cu}(\phi_{ij})(c)} I_{y_j}}}
& {\,\,\underline{\mathrm{K}}(I_{y_j}) \,\,}  \ar[ru]^-{}}
$$

Note that
$$
{\rm \underline{K}}(\phi_{{i\infty}{I_{x_i}I_z}})(0,\mathfrak{u}_{i},\oplus_{p=1}^{\infty}(\mathfrak{s}_{i,p},\mathfrak{s}^{i,p}))
=
{\rm \underline{K}}(\phi_{{j\infty}{I_{y_j}I_z}})(0,\mathfrak{v}_{j},\oplus_{p=1}^{\infty}(\mathfrak{t}_{j,p},\mathfrak{t}^{j,p})\in {\rm \underline{K}}(I_z),
$$
then there exists $m\geq m_c$ such that
$$
\delta_{I_{{\rm Cu}(\phi_{ij})(c)} I_{y_j}}\circ\underline{\rm K}(\phi_{{ij}I_{c} I_{{\rm Cu}(\phi_{ij})(c)}})(0,\mathfrak{w}_i,\oplus_{p=1}^{\infty}(\mathfrak{r}_{i,p},\mathfrak{r}^{i,p}))
$$$$=
(0,\mathfrak{v}_{j},\oplus_{p=1}^{\infty}(\mathfrak{t}_{j,p},\mathfrak{t}^{j,p}))\in {\rm  \underline{K}}(I_{y_j})
$$
holds for any $j\geq m$.

  Recall that $c\ll y_j$, by Theorem \ref{jin bijiao}, we have
$${\rm \underline{Cu}} (\phi_{ij})(c,\mathfrak{w}_i,\oplus_{p=1}^{\infty}(\mathfrak{r}_{i,p},\mathfrak{r}^{i,p}))\ll
(y_j,\mathfrak{v}_{j},\oplus_{p=1}^{\infty}(\mathfrak{t}_{j,p},\mathfrak{t}^{j,p}))
$$
holds for any $j\geq m$. Then we get $f\leq g.$

Similarly, $(y_i)\leq (x_i)$ will imply $g\leq f$, then $\omega$ is injective.

Now we have $\omega:\, S\cong{\rm \underline{Cu}}(A)$ as objects in the category ${\rm {Cu}}^\sim$.
By  Theorem \ref{k-total continuous}, Theorem \ref{Cu total thm} and Proposition \ref{cu total con},
we have
$${\rm Gr}(S_c)\cong\lim_{\longrightarrow}{\rm \underline{K}}(A_i)={\rm \underline{K}}(A).$$
Still by Theorem \ref{Cu total thm}, we have ${\rm Gr}({\rm \underline{Cu}}(A)_c)$ is also canonically isomorphic with ${\rm \underline{K}}(A)$.
Then Gr($\omega_c$) works as the identity map on ${\rm \underline{K}}(A)$, which is $\Lambda$-linear, and hence, we have $S\cong{\rm \underline{Cu}}(A)$  as objects in the category ${\rm \underline{Cu}}$. This completes the proof.

\end{proof}

\section{Recover the total K -theroy}
In this  section, we show that ${\rm \underline{ Cu}}$ contains more information than $\underline{\rm K}$, hence, ${\rm \underline{ Cu}}$ is a complete invariant for certain real rank zero algebras.

\begin{notion}\rm
We say $(S,u)$ is a ${\rm \underline{Cu}}$-$semigroup$ $with$ $order$-$unit$ if $(S,u)$ is a positively directed ${\rm \underline{Cu}}$-semigroup satisfying $\rho(S_c)\cap \{-\rho(S_c)
 \}=\{0\}$ with a compact order-unit. Now a ${\rm \underline{Cu}}$-morphism $f:S\rightarrow T$ with $f(u)\leq v$ will be called a ${\rm \underline{Cu}}_u$-morphism between $(S,u)$ and $(T,v)$.
Denote ${\rm \underline{Cu}}_u$
as the category whose objects are ${\rm \underline{Cu}}$-semigroups with order-unit and morphisms are
${\rm \underline{Cu}}_u$-morphisms.
\end{notion}

After adding a unit,  we transform  Theorem \ref{Cu total thm} into the following:

\begin{lemma}
  The assignment
  \begin{eqnarray*}
    {\rm \underline{Cu}}_u:\,  C_{sr1}^* & \rightarrow & {\rm \underline{Cu}}_u  \\
    A &\mapsto & ({\rm \underline{Cu}}(A),([1_A],0,\oplus_{p=1}^{\infty}(0,0))) \\
    \phi &\mapsto & {\rm \underline{Cu}}(\phi)
  \end{eqnarray*}
  from the category of unital, separable $C^*$-algebras of stable rank one to the category ${\rm \underline{Cu}}^\sim$ is a covariant functor.
\end{lemma}
\begin{proof}
  Note that $[1_A]$ is an order unit of
  ${\rm \underline{K}}(A)_+$.

  For any
  $
  (x,\mathfrak{u},\oplus_{p=1}^{\infty}(\mathfrak{s}_p,\mathfrak{s}^p))\in
  {\rm \underline{Cu}}(A),
  $
  $$
  (x,\mathfrak{u},\oplus_{p=1}^{\infty}(\mathfrak{s}_p,\mathfrak{s}^p))+(x,-\mathfrak{u},\oplus_{p=1}^{\infty}(-\mathfrak{s}_p,-\mathfrak{s}^p))
  =(2x,0,\oplus_{p=1}^{\infty}(0,0))\geq 0
  $$
  implies that ${\rm \underline{Cu}}(A)$ is positively directed.

  From  Theorem \ref{Cu total thm}, for any unital, separable $C^*$-algebra $A$ with stable rank one, we have
 $$
 ({\rm Gr}({\rm \underline{Cu}}(A)_c),\rho({\rm \underline{Cu}}(A)_c), ([1_A],0,\oplus_{p=1}^{\infty}(0,0)))\cong({\rm \underline{K}}(A),{\rm \underline{K}}(A)_+,[1_A]),
 $$
where $\rho:\, {\rm \underline{Cu}}(A)_c\to {\rm Gr}({\rm \underline{Cu}}(A)_c)$ is the natural map $(\rho(x)=[(x,0)]).$
From Theorem \ref{ordertotal}, we have
 $\rho({\rm \underline{K}}(A)_+)\cap \{-\rho({\rm \underline{K}}(A)_+)
 \}=\{0\},$ then
  $\rho({\rm \underline{Cu}}(A)_c)\cap \{-\rho({\rm \underline{Cu}}(A)_c)
 \}=\{0\}.$

In conclusion, we have $({\rm \underline{Cu}}(A),([1_A],0,\oplus_{p=1}^{\infty}(0,0)))$ is a ${\rm \underline{Cu}}$-semigroup with order-unit
and ${\rm \underline{Cu}}(\phi)$ is a ${\rm \underline{Cu}}_u$-morphism.

\end{proof}
Now we can recover the total K-theory from the total Cuntz semigroup. 
\begin{proposition}\label{recover prop}
  The assignment
\begin{eqnarray*}
  \underline{H}:\,  {\rm \underline{Cu}}_u &\rightarrow& \Lambda-{\rm module} \\
  (S,u) &\mapsto & ({\rm Gr}(S_c),\rho(S_c),[u])\\
  \phi &\mapsto & {\rm Gr}(\phi_c),
\end{eqnarray*}
is a functor, where $\rho:\, S_c\to {\rm Gr}(S_c)$ is the natural map ($\rho(x)=[(x,0)])$.

 The functor $\underline{H}$ yields a natural isomorphism $\underline{H}\circ {\rm \underline{Cu}}_u\simeq  {\rm \underline{K}}$, which means, for any unital, separable $C^*$-algebras $A,B$ with stable rank one, if
$$
({\rm \underline{Cu}}(A),([1_A],0,\oplus_{p=1}^{\infty}(0,0)))\cong
({\rm \underline{Cu}}(B),([1_B],0,\oplus_{p=1}^{\infty}(0,0))),
$$
then
$$
 ({\rm \underline{K}}(A),{\rm \underline{K}}(A)_+,[1_A])\cong ({\rm \underline{K}}(B),{\rm \underline{K}}(B)_+,[1_B]).
$$
\end{proposition}
\begin{proof}
Let $(S, u) \in \mathrm{\underline{Cu}}_{u}$. Then $({\rm Gr}(S_c),\rho(S_c),[u])$ is a $\Lambda$-module with order-unit. Now let $\phi: S \longrightarrow T$ be a $\mathrm{\underline{Cu}}_{u}$-morphism between two $\mathrm{\underline{Cu}}$-semigroups with order-unit $(S, u),(T, v)$.
It follows that $\phi_{c}: S_{c} \longrightarrow T_{c}$ is a ${\rm Mon}_\leq$-morphism, and the induced Grothendieck map ${\rm Gr}\left(\phi_{c}\right): {\rm Gr}\left(S_{c}\right) \longrightarrow {\rm  Gr}\left(T_{c}\right)$ is $\Lambda$-linear such that ${\rm Gr}\left(\phi_{c}\right)\left(S_{c}\right) \subseteq T_{c}$. Finally, using that $\phi(u) \leq \phi(v)$, we obtain ${\rm Gr}\left(\phi_{c}\right)(u) \leq v$. We conclude that $\underline{H}$ is a well-defined functor.

From Theorem \ref{ordertotal} and  Theorem \ref{Cu total thm}, for any $A\in C_{sr1}^*$, we have
 $$\alpha^*_A:\,
 ({\rm Gr}({\rm \underline{Cu}}(A)_c),\rho({\rm \underline{Cu}}(A)_c), ([1_A],0,\oplus_{p=1}^{\infty}(0,0)))\cong({\rm \underline{K}}(A),{\rm \underline{K}}(A)_+,[1_A]).
 $$
This means we do recover ${\rm \underline{K}}$ from ${\rm \underline{Cu}}$ as we anticipated.

\end{proof}

In general, $\underline{H}$ is not faithful, but it can be faithful if we restrict the domain of $\underline{H}$ to a suitable subcategory of $C_{sr1}^*$. To achieve a total version of Theorem \ref{Kthm}, we list the following.
\begin{definition}\rm
We say that an ideal $I$ in a $C^*$-algebra $A$ is $\mathrm{K}$-$pure$ if both sequences
$$0\to  \mathrm{K}_\delta(I ) \to\mathrm{K}_\delta(A)\to \mathrm{K}_\delta(A/I ) \to 0,\quad \delta=0,1$$
are pure exact. We say a $C^*$-algebra $A$ is $\mathrm{K}$-$pure$, if all ideals in $A$ are $\mathrm{K}$-$pure$.
\end{definition}
\begin{lemma}\label{k-pure lemma}
If an ideal $I$ in a $C^*$-algebra $A$ is $\mathrm{K}$-$pure$,
then
for any $p \geq 2$, the sequences
$$0 \to {\rm K}_*(I ; \mathbb{Z}_p) \to {\rm K}_*(A; \mathbb{Z}_p) \to {\rm K}_*(A/I ; \mathbb{Z}_p) \to 0$$
are exact.
\end{lemma}
\begin{proof}
For  any $p \geq 2$, we have the following well-known commutative diagram induced by the natural embedding  $\iota:\,I\to A$.
$$
\xymatrixcolsep{3pc}
\xymatrix{
{\mathrm{K}_\delta(I)} \ar[d]_-{\mathrm{K}_\delta(\iota)} \ar[r]^-{\times p}
& {\mathrm{K}_\delta(I)} \ar[d]_-{\mathrm{K}_\delta(\iota)} \ar[r]^-{\rho_\delta^I}
& {\mathrm{K}_\delta(I; \mathbb{Z}_p)} \ar[d]_-{\mathrm{K}_\delta(\iota;\mathbb{Z}_p)} \ar[r]^-{}
& 
 {\mathrm{K}_{1-\delta}(I)} \ar[d]_-{\mathrm{K}_{1-\delta}(\iota)}
 \\
{\mathrm{K}_\delta(A)} \ar[r]_-{\times p}
& {\mathrm{K}_\delta(A)} \ar[r]_-{\rho_\delta^A}
& {\mathrm{K}_\delta(A; \mathbb{Z}_p)} \ar[r]_-{}
& 
{\mathrm{K}_{1-\delta}(A)}.
}
$$
By assumption, we have  $\mathrm{K}_\delta(\iota)$ are injective for $\delta=0,1$.
Note that the following diagram is exact, hence, what we need to show is $\mathrm{K}_\delta(\iota;\mathbb{Z}_p)$ are injective maps for both $\delta=0,1$.
$$
\xymatrixcolsep{3pc}
\xymatrix{
{\mathrm{K}_\delta(I; \mathbb{Z}_p)}  \ar[r]^-{\mathrm{K}_\delta(\iota;\mathbb{Z}_p)}
& {\mathrm{K}_\delta(A; \mathbb{Z}_p)} \ar[r]^-{}
& 
 {\mathrm{K}_\delta(A/I; \mathbb{Z}_p)} \ar[d]_-{}
 \\
{\mathrm{K}_{1-\delta}(A/I; \mathbb{Z}_p)} \ar[u]_-{}
& {\mathrm{K}_{1-\delta}(A; \mathbb{Z}_p)} \ar[l]_-{}
& 
{\mathrm{K}_{1-\delta}(I; \mathbb{Z}_p)} \ar[l]_-{\mathrm{K}_{1-\delta}(\iota;\mathbb{Z}_p)}.
}
$$
Given $a_1\in\mathrm{K}_\delta(I; \mathbb{Z}_p)$ with $\mathrm{K}_\delta(\iota;\mathbb{Z}_p)(a_1)=0$,
the standard diagram chasing gives $a_2\in\mathrm{K}_\delta(I)\subset \mathrm{K}_\delta(A)$ and $a_3\in \mathrm{K}_\delta(A)$ such that
$$
\rho_\delta^I(a_2)=a_1\quad {\rm and}\quad
p\times a_3=a_2.
$$
Since
$$0\to  \mathrm{K}_\delta(I ) \to\mathrm{K}_\delta(A)\to \mathrm{K}_\delta(A/I ) \to 0,\quad \delta=0,1$$
are pure exact, i.e., for $p\geq 2$ we have
$$
p\times \mathrm{K}_\delta(I )=\mathrm{K}_\delta(I )\cap (p\times\mathrm{K}_\delta(A )),
$$
there is an $a_4\in\mathrm{K}_\delta(I )$ such that $p\times a_4= a_2$. By exactness, we have $a_1=0$.

\end{proof}
\begin{notion}\rm (\cite[2.2-2.3]{GJL})
 An A$\mathcal{HD}$ algebra is an inductive limit of finite direct sums of the form $M_n(\mathbb{I}_p^\sim)$ and $PM_n(C(X))P$, where $\mathbb{I}_p$ is the Elliott-Thomsen dimension-drop interval algebra
and $X$ is one of the following finite connected CW complexes: $\{pt\},~\mathbb{T},~[0, 1],~T_{II,k}.$ $P\in M_n(C(X))$ is a projection and $T_{II,k}$ is the 2-dimensional connected simplicial complex with $H^1(T_{II,k})=0$ and $H^2(T_{II,k})=\mathbb{Z}/k\mathbb{Z}$. (In \cite{DG}, this class is called ASH algebras.)
\end{notion}

\begin{proposition}{\rm (}\cite[Proposition 4.4]{DE}{\rm )}\label{de prop}
Suppose an extension
$$0 \to I \to  A \to  A/I \to 0$$
is given. If $A$ is an A$\mathcal{HD}$ algebra of real rank zero, we have

(i) $I$ and $A/I$ are A$\mathcal{HD}$ algebras of real rank zero.

(ii) The ideal $I$ in $A$ is $\mathrm{K}$-$pure$.

(iii) For any $p \geq 2$, the sequences
$$0 \to {\rm K}_*(I ; \mathbb{Z}_p) \to {\rm K}_*(A; \mathbb{Z}_p) \to {\rm K}_*(A/I ; \mathbb{Z}_p) \to 0$$
are pure exact.
\end{proposition}
That is, all  A$\mathcal{HD}$ algebras of real rank zero are K-pure.
\begin{theorem}{\rm (}\cite[Proposition 4.12]{DG}{\rm )}\label{gong ideal preserving}
Let $A$, $B$ be separable $C^*$-algebras of real rank zero and
stable rank one. Let $\phi:\, {\rm \underline{K}}(A)\to {\rm \underline{K}}(B)$ be a morphism of $\mathbb{Z}_2\times \mathbb{Z}^+$-graded
groups. The following are equivalent.

(i) $\phi({\rm \underline{K}}(A)_+)\subset{\rm \underline{K}}(B)_+.$

(ii) $\phi({\rm {K}}_0^+(A))\subset{\rm {K}_0^+}(B)$ and $\phi(\underline{\mathrm{K}}_{I}(A)) \subset
\underline{\mathrm{K}}_{I'}(B)$ for all ideals $I$ in $A$.
\end{theorem}

\begin{theorem}\label{totalKthm}\label{total-Kthm}
  Upon restriction to the class of unital, separable $\mathrm{K}$-$pure$ (or A$\mathcal{HD}$) $C^*$-algebras of stable rank one and real rank zero, there are natural equivalences of functors:
    $$
  \underline{H}\circ \underline{\rm Cu}_{u}\simeq \underline{{\rm K}}\quad{ and}\quad   \underline{G}\circ \underline{{\rm K}}\simeq \underline{\rm Cu}_{u}.
 $$
  Therefore, for these algebras, ${\rm \underline{K}}$ is a classifying functor if, and only if, so is ${\rm \underline{Cu}}_{u}$.
\end{theorem}
\begin{proof}
The proof is as same as we did for $\mathrm{K}_*$ and $\mathrm{Cu}_1$.
Note that the domain will be the full subcategory of ${\rm \underline{Cu}}_{u}$ whose objects are all those $({\rm \underline{Cu}}(A),([1_A],0,\oplus_{p=1}^{\infty}(0,0)))$ ($A$ is a unital, separable, $\mathrm{K}$-$pure$ (or A$\mathcal{HD}$) algebra with stable rank one  and real rank zero.), while the codomain is category of all the ${\rm \underline{K}}$-invariants for the same class of $C^*$-algebras.
We only need to show that the corresponding restriction functor of $\underline{H}$, which we will still call $\underline{H}$, is a full, faithful and dense functor.

  It was shown in Theorem \ref{Cu total thm} that
  $$
  \alpha: {\rm \underline{Cu}}(A)_c\rightarrow {\rm \underline{K}}(A)_+,
  $$
  $$
  ([e],\mathfrak{u},\oplus_{p=1}^{\infty}(\mathfrak{s}_p,\mathfrak{s}^p))\mapsto ([e],\delta_{I_eA}(\mathfrak{u},\oplus_{p=1}^{\infty}( \mathfrak{s}_p,\mathfrak{s}^p)))
  $$
  is surjective. We will show that $\alpha$ is injective.

 Suppose that
  $$
  \alpha([e],\mathfrak{u},\oplus_{p=1}^{\infty}
  (\mathfrak{s}_p,\mathfrak{s}^p))=\alpha([q],\mathfrak{v},\oplus_{p=1}^{\infty}(\mathfrak{t}_p,\mathfrak{t}^p)),
  $$
which is
  $$
  ([e],\delta_{I_eA}(\mathfrak{u},\oplus_{p=1}^{\infty}( \mathfrak{s}_p,\mathfrak{s}^p)))=
  ([q],\delta_{I_pA}(\mathfrak{v},\oplus_{p=1}^{\infty}( \mathfrak{t}_p,\mathfrak{t}^p))).
  $$

  Since $A$ is $\mathrm{K}$-pure, for any ideal $I$ of $A$, Lemma \ref{k-pure lemma} shows that the group $\mathrm{K}_*(I ; \mathbb{Z}_p)$ is a sub-group of $\mathrm{K}_*(A ; \mathbb{Z}_p)$, for all $ p\geq 2$, then
  both $\delta_{I_eA}$ and $\delta_{I_pA}$ are the same injective map.
Then we have
   $$
  ([e],\mathfrak{u},\oplus_{p=1}^{\infty}(\mathfrak{s}_p,\mathfrak{s}^p))=([q],\mathfrak{v},\oplus_{p=1}^{\infty}(\mathfrak{t}_p,\mathfrak{t}^p)).
  $$
  Thus $\alpha$ is an ordered isomorphism. Combining this with Proposition \ref{recover prop}, we have
   $$ \alpha^*_A\,:\,
   \underline{H}({\rm \underline{Cu}}(A)_c, ([1_A],0,\oplus_{p=1}^{\infty}(0,0)))\cong({\rm \underline{K}}(A),{\rm \underline{K}}(A)_+,[1_A]).
 $$
  This means that $\underline{H}$ is dense.

  To show that $\underline{H}$ is faithful. Let $\phi,\psi:  {\rm \underline{Cu}}(A)\longrightarrow {\rm \underline{Cu}}(B)$ be two $\mathrm{\underline{Cu}}_{u}$-morphisms such that $ \underline{H}(\phi)=\underline{H}(\psi)$. Then ${\rm Gr}(\phi_c)={\rm Gr}(\psi_c)$, and hence,
$$
{\rm Gr}(\phi_c)\mid_{\rho_A({\rm \underline{Cu}}(A)_c)}={\rm Gr}(\psi_c)\mid_{\rho_A({\rm \underline{Cu}}(A)_c)}:\,\rho_A({\rm \underline{Cu}}(A)_c)\to\rho_B({\rm \underline{Cu}}(B)_c),
$$
where $\rho_A$ and $\rho_B$ are
the corresponding maps:
$$
\rho_A: {\rm \underline{Cu}}(A)_c\rightarrow\rho_A({\rm \underline{Cu}}(A)_c)\quad{\rm and}\quad
\rho_B: {\rm \underline{Cu}}(B)_c\rightarrow\rho_B({\rm \underline{Cu}}(B)_c).
$$
Recall that ${\rm \underline{K}}(A)_+$ and  ${\rm \underline{K}}(B)_+$ are subsets of ${\rm \underline{K}}(A)$ and ${\rm \underline{K}}(B)$, respectively. Combining this with ${\rm \underline{Cu}}(A)_c\cong {\rm \underline{K}}(A)_+$ and ${\rm \underline{Cu}}(B)_c\cong {\rm \underline{K}}(B)_+$, we have $\rho_A$ and $\rho_B$ are injective, which implies $\phi_c=\psi_c$. Note that any morphism between algebraic ${\rm Cu}^\sim$-semigroups is entirely determined by its restriction to compact elements, then we get $\phi=\psi$.

  Now we prove that $\underline{H}$ is full. Suppose we have an ordered morphism
  $$
  \phi: ({\rm \underline{K}}(A),{\rm \underline{K}}(A)_+,[1_A])\rightarrow ({\rm \underline{K}}(B),{\rm \underline{K}}(B)_+,[1_B]),
  $$
  which is $\Lambda$-linear. Then
  $$
  \phi|_{{\rm \underline{K}}(A)_+}: {\rm \underline{K}}(A)_+\rightarrow {\rm \underline{K}}(B)_+, \,\,\phi([1_A])\leq [1_B],
  $$
  and
  $$
  \phi|_{{\rm \underline{K}}(A)_+}({\rm K}_0(A)_+)\subset {\rm K}_0(B)_+.
  $$

  By the functoriality of ${\rm Cu}^\sim$ (see \ref{cucomplete}), $\phi|_{{\rm \underline{K}}(A)_+}$ induces a ${\rm Cu}_u^\sim$-morphism
  \begin{align*}
    \gamma^\sim(\phi|_{{\rm \underline{K}}(A)_+})  \,:\, & ({\rm Cu}^\sim({\rm \underline{K}}^+(A)), ([1_A],0,\oplus_{p=1}^{\infty}(0,0))) \\
     & \rightarrow({\rm Cu}^\sim({\rm \underline{K}}^+(B)), ([1_B],0,\oplus_{p=1}^{\infty}(0,0))),
  \end{align*}
  where the order on ${\rm \underline{K}}(A)_+$ is induced by ${\rm K}_0^+(A)$.

  By the above results and Proposition \ref{alg completion}, we have
  $$
  {\rm \underline{Cu}}(A)\cong {\rm Cu}^{\sim}({\rm \underline{Cu}}(A)_c)
  \cong{\rm Cu}^{\sim}({\rm \underline{K}}(A)_+).
  $$
  For each 
   $A$ in our assumption, denote  by $i_A :{\rm \underline{Cu}}(A)\rightarrow {\rm Cu}^{\sim}({\rm \underline{K}}(A))$ the canonical ${\rm Cu}^\sim$-isomorphism
  and denote by $
 \alpha_A: {\rm \underline{Cu}}(A)_c\rightarrow {\rm \underline{K}}(A)_+
  $ the canonical ordered monoid isomophism.

  Then
  \begin{align*}
    i_B^{-1}\circ\gamma^\sim(\phi|_{{\rm \underline{K}}(A)_+})\circ i_A :\, &
    ({\rm \underline{Cu}}(A),([1_A],0,\oplus_{p=1}^{\infty}(0,0))) \\
     & \rightarrow ({\rm  \underline{Cu}}(B), ([1_B],0,\oplus_{p=1}^{\infty}(0,0)))
  \end{align*}
  is an ordered ${\rm Cu}^\sim$-morphism.
  After identifying ${\rm \underline{Cu}}(A)_c$ and ${\rm \underline{Cu}}(B)_c$ with ${\rm \underline{K}}(A)_+$ and $ {\rm \underline{K}}(B)_+$ through
  $\alpha_A$ and $\alpha_B,$ respectively, we have
  $$
  (i_B^{-1}\circ\gamma^\sim(\phi|_{{\rm \underline{K}}(A)_+})\circ i_A)_c=\phi|_{{\rm \underline{K}}(A)_+}.
  $$

  Note that ${\rm \underline{Cu}}(A),$ ${\rm \underline{Cu}}(B)\in {\rm \underline{Cu}}$ and
$
    {\rm Gr}(i_B^{-1}\circ\gamma^\sim(\phi|_{{\rm \underline{K}}(A)_+})\circ i_A)_c)=\phi
$
  is $\Lambda$-linear by assumption.  Then $i_B^{-1}\circ\gamma^\sim(\phi|_{{\rm \underline{K}}(A)_+})\circ i_A$ is a ${\rm \underline{Cu}}_u^\sim$-morphism.

  Using the functoriality of $\underline{H}$,  we obtain
  $$
  \underline{H}(i_B^{-1}\circ\gamma^\sim(\phi|_{{\rm \underline{K}}(A)_+})\circ i_A)={\rm Gr}(i_B^{-1}\circ\gamma^\sim(\phi|_{{\rm \underline{K}}(A)_+})\circ i_A)_c= {\rm Gr}(\phi|_{{\rm \underline{K}}(A)_+})=\phi.
  $$
  Then $\underline{H}$ is full.

 Therefore, by standard category theory, there exists a functor $\underline{G}$ such that $\underline{H}\circ \underline{G}$ and $\underline{G}\circ \underline{H}$ are naturally equivalent to the respective identities. Then we have
  $$
  \underline{H}\circ \underline{\rm Cu}_{u}\simeq \underline{{\rm K}}\quad{\rm and}\quad   \underline{G}\circ \underline{{\rm K}}\simeq \underline{\rm Cu}_{u}.
 $$\end{proof}
\begin{corollary}
  By restricting to the category ${\rm \underline{Cu}}_u$, we can recover ${\rm \underline{K}}$ from ${\rm \underline{Cu}}_u$ through ${ \underline{H}}$. A fortiori, we have ${\rm \underline{Cu}}_u$ is a complete invariant for unital A$\mathcal{HD}$ algebras of real rank zero
 with slow dimension growth (see Theorem 9.1 in \cite{DG}).
\end{corollary}
\begin{remark}
Suppose that
$\phi:\, {\rm \underline{K}}(A)\to {\rm \underline{K}}(B)$ is a isomorphism of $\mathbb{Z}_2\times \mathbb{Z}^+$-graded
groups satisfying
$\phi({\rm \underline{K}}(A))_+={\rm \underline{K}}(B)_+.$ Then by Theorem \ref{gong ideal preserving}, we get $\phi({\rm {K}}_0^+(A))={\rm {K}_0^+}(B)$ and $\phi(\underline{\mathrm{K}}_{I}(A)) =
\underline{\mathrm{K}}_{J}(B)$ for each ideal $I$ in $A$ and the relative $J$ in $B$.

Note that for any ideal $I$ of $A$, Lemma \ref{k-pure lemma} shows that the map $\mathrm{K}_*(I ; \mathbb{Z}_p)$ is a sub-group of $\mathrm{K}_*(A ; \mathbb{Z}_p)$, for all $ p\geq 2$. So do $J$ and $B$. Thus we have the mod-$p$ diagrams commutative for all $p\geq 2$ and $\delta=0,1$.

  \begin{displaymath}
\xymatrixcolsep{1.2pc}
\xymatrix{
{\rm K}_\delta(I) \ar[rr]\ar[dr]\ar@{=}[ddd] && {\rm K}_\delta(I;\mathbb{Z}_n) \ar[rr]\ar[dr]\ar@{=}[ddd]  && {\rm K}_{1-\delta}(I) \ar[dr]\ar@{=}[ddd]\\
&{\rm K}_\delta(A) \ar@{=}[d]\ar[rr]&&{\rm K}_{\delta}(A;\mathbb{Z}_n) \ar@{=}[d]\ar[rr]&&{\rm K}_{1-\delta}(A)\ar@{=}[d]\\
&{\rm K}_\delta(B) \ar[rr]&&{\rm K}_\delta(B;\mathbb{Z}_n) \ar[rr]&&{\rm K}_{1-\delta}(B)  \\
{\rm K}_\delta(J) \ar[rr]\ar[ur] && {\rm K}_\delta(J;\mathbb{Z}_n) \ar[rr]\ar[ur]   && {\rm K}_{1-\delta}(J) \ar[ur]
}
\end{displaymath}
\end{remark}


\begin{remark}
For any unital, separable $C^*$-algebras with stable rank one and real rank zero,  the unitary Cuntz semigroup can be recovered functorially from the invariant ${\rm K}_*$ (Theorem \ref{Kthm}). But we don't know whether the total version Theorem \ref{total-Kthm} is still true or not, if we delete the assumption of ``$\mathrm{K}$-$pure$ (or  A$\mathcal{HD}$ algebras)''.

The key point is that the sequences
$$0 \to {\rm K}_i(I )\to {\rm K}_i(A) \to {\rm K}_i(A/I ) \to 0$$
are always exact;
while for any $p \geq 2$, we are not sure whether the sequences
$$0 \to {\rm K}_*(I ; \mathbb{Z}_p) \to {\rm K}_*(A; \mathbb{Z}_p) \to {\rm K}_*(A/I ; \mathbb{Z}_p) \to 0$$
are exact or not.

These above exact sequences will lead the commutativity of mod-$p$ diagram for $\delta=0,1$,  which will generalize Theorem \ref{total-Kthm}.

One more observation is that if we replaced ``$\mathrm{K}$-pure (or A$\mathcal{HD}$ algebras)'' by ``UCT'', then by \cite{RS}, we have
$$
{\rm {K}_*}(I; \mathbb{Z}_p)\cong{\rm {K}_*}(J; \mathbb{Z}_p).
$$
The problem is that we don't know whether this isomorphism is compatible with the other maps in the above mod-$p$ diagram.

At last, we have the following question.
\end{remark}
\begin{question}\label{q1}
 Let $A$ be a unital, separable $C^*$-algebra of stable rank one and real rank zero, does $({\rm\underline{K}}(A),{\rm\underline{K}}(A)_+)$ determine ${\rm\underline{Cu}}(A)$?
\end{question}

Note that if we replace the definition of ${\rm\underline{Cu}}(A)$  by
$$
{\rm \underline{Cu}}^I(A)\triangleq \coprod_{I\in {\rm Lat}_f(A)} {\rm Cu}_f(I)\times {\rm K}_1^{I}(A)\times\bigoplus_{p=1}^{\infty} {\rm K}_*^I (A; \mathbb{Z}_p),
$$
where
${\rm K}_1^{I}(A)$ and  ${\rm K}_*^I (A; \mathbb{Z}_p)$ are images of ${\rm K}_1(I)$ and ${\rm K}_* (I; \mathbb{Z}_p)$ in ${\rm K}_1(A)$ and ${\rm K}_* (A, \mathbb{Z}_p)$, respectively.

We point that ${\rm\underline{Cu}}^I(A)$ is also a continuous invariant, it is routine to get that
$
{\rm \underline{Cu}}^I(A)$ and $({\rm\underline{K}}(A),{\rm\underline{K}}(A)_+)
$
are equivalent invariants for any unital, separable $C^*$-algebras of stable rank one and  real rank zero.
And  for both simple case and  A$\mathcal{HD}$ of real rank zero case, it is obvious that $
{\rm \underline{Cu}}^I$ and $
{\rm \underline{Cu}}$ coincide with each other.

As supplementary, we raise one more question.

\begin{question}\label{q2}
  Let $A$ be a unital,  separable $C^*$-algebra of stable rank one, when do we have ${\rm \underline{Cu}}^I(A)$ and ${\rm \underline{Cu}}(A)$ determine each other?
\end{question}

Under the real rank zero setting, an affirmative answer to Question \ref{q2} will give an affirmative answer to Question \ref{q1}. But in general, these questions remain open for us.

\section*{Acknowledgements}
The research of first author was supported by NNSF of China (No.:12101113,
No.:11920101001) and the Fundamental Research Funds for the Central Universities (No.:2412021QD001). The second author was supported by NNSF of China (No.:12101102).

\end{document}